%% file: main.tex
\documentclass[10pt,preprint]{elsarticle}

\usepackage[margin=1in,papersize={8.5in,11in}]{geometry}

\newtheorem{theorem}{\bf Theorem}[section]

\newenvironment{proof}{{\noindent \bf \em Proof:}}{\hfill$\square$}
\usepackage{hyperref}
\usepackage{amsmath}
\usepackage{cleveref}

\input{shared}

\title{Inexact subspace projection methods for low-rank tensor eigenvalue problems} 

\begin{document}
\begin{frontmatter}

\author[lbnl]{Alec Dektor\corref{correspondingAuthor}}
\ead{adektor@lbl.gov}

\author[vt]{Peter DelMastro}

\author[lbnl]{Erika Ye}
\author[lbnl]{Roel Van Beeumen}
\author[lbnl]{Chao Yang}

\address[lbnl]{Applied Mathematics and Computational Research Division,
  Lawrence Berkeley National Laboratory, Berkeley, CA}

\address[vt]{Department of Mathematics, Virginia Tech, Blacksburg, VA}

\cortext[correspondingAuthor]{Corresponding author}

\journal{Arxiv}
\date{}

\begin{abstract}
We propose inexact subspace iteration for solving high-dimensional eigenvalue problems with low-rank structure. Inexactness stems from low-rank compression, enabling efficient representation of high-dimensional vectors in a low-rank tensor format. A primary challenge in these methods is that standard operations, such as matrix-vector products and linear combinations, increase tensor rank, necessitating rank truncation and hence approximation. We compare the proposed methods with an existing inexact Lanczos method with low-rank compression. This method constructs an approximate orthonormal Krylov basis, which is often difficult to represent accurately in low-rank tensor formats, even when the eigenvectors themselves exhibit low-rank structure. In contrast, inexact subspace iteration uses approximate eigenvectors (Ritz vectors) directly as a subspace basis, bypassing the need for an orthonormal Krylov basis. Our analysis and numerical experiments demonstrate that inexact subspace iteration is much more robust to rank-truncation errors compared to the inexact Lanczos method. We also demonstrate that rank-truncated subspace iteration can converge for problems where the DMRG method stagnates. Furthermore, the proposed subspace iteration methods do not require a Hermitian matrix, in contrast to Lanczos and DMRG, which are designed specifically for Hermitian matrices.
\end{abstract}



\end{frontmatter}

\section{Introduction}
\label{sec:Intro}

We are interested in computing several extreme eigenpairs 
\begin{equation} \label{eq:eigen}
A \psi = \lambda \psi, 
\end{equation}
of a matrix $A \in \mathbb{C}^{(n_1 \cdots n_d) \times (n_1 \cdots n_d)}$ that can be written as a sum of Kronecker products 
\begin{equation} \label{eq:sum_of_krons}
    A = \sum_{j = 1}^{r_A} A_1^{(j)} \otimes A_2^{(j)} \otimes \cdots \otimes A_d^{(j)}, 
\end{equation} 
where $r_A$ is the separation (or canonical polyadic) rank of $A$ and $A_{i}^{(j)} \in \mathbb{C}^{n_i \times n_i}$ for all $i = 1,2,\ldots,d$ and $j = 1,\ldots,r_A$. Examples of \eqref{eq:sum_of_krons} appear naturally in quantum many-body physics and arise in many other applications from the discretization of partial differential operators on separable domains. The number of entries in $A$ and its eigenvectors $\psi \in \mathbb{C}^{n_1 \cdots n_d}$ grows exponentially fast with the parameter $d$, making classical numerical methods too expensive when directly applied to $A$, even for moderate $d$. Given the tensor product structure of the operator \eqref{eq:sum_of_krons}, it is natural to seek eigenvectors in a compact form commensurate with the tensor product structure of the operator $A$. Many tensor networks (or formats) are available for representing operators and vectors \cite{Grasedyck2013}, offering significant reductions in the degrees of freedom required, provided the objects can be approximated with a network of relatively small tensor rank. For quantum many-body Hamiltonians with separated eigenvalues and local interactions on 1D geometries, it is well-known that the eigenvectors corresponding to the smallest eigenvalues admit accurate low-rank tensor train (or matrix product state) approximations, thanks to the area law \cite{Hastings2007}. For other applications, e.g., the spatial discretization of high-dimensional elliptic operators, it has been demonstrated that eigenvectors can be well-approximated in low-rank hierarchical tensor formats \cite{ballani2013projection,hackbusch2012}. 

A widely used approach for computing the extreme eigenpairs of a Hermitian matrix $A$, expressed in the form \eqref{eq:sum_of_krons}, is the density matrix renormalization group (DMRG) method \cite{Schollwock2005,Schollwock2011,White1993}, originally developed in the physics community. The DMRG method represents the approximate eigenvector as a tensor network and optimizes one (or two) low-dimensional tensors in the network at a time by projecting $A$ onto subspaces defined by the remaining tensors. This optimization process cycles through all tensors in the network, often multiple times, until convergence. The method achieves computational efficiency by reducing the high-dimensional eigenvalue problem to a sequence of low-dimensional eigenvalue computations. However, the alternating optimization framework of DMRG is prone to convergence to local minima for certain types of systems \cite{White2005}. 
To address this issue, several variations of the DMRG method have been proposed, such as subspace enrichment techniques that increase the rank during optimization \cite{Dolgov2015}. Despite these enhancements, the DMRG algorithm and its variants remain unreliable for certain systems. Furthermore, the original DMRG formulation was designed to compute only the lowest eigenpair of $A$. Extensions for computing multiple eigenpairs often rely on sequential deflation, where eigenvectors are computed one at a time, or on block low-rank representations \cite{dolgov2014,kressner2014,HuChan2015,MBLExcitedDMRG}. In our experience, the latter can require higher ranks than are needed to represent the eigenvectors individually, leading to inefficiencies. Also, DMRG is based on an energy minimization principle that requires $A$ to be Hermitian. This requirement is not imposed by the methods proposed in this paper.

We propose an inexact polynomial-filtered subspace iteration for computing several extreme eigenpairs of $A$. We employ a low-rank tensor representation of $A$, such as the CP representation in \eqref{eq:sum_of_krons} or TT-matrix representation, and represent all vectors in a compatible low-rank tensor format. A key challenge in applying this iterative method is that standard operations, such as matrix-vector products and vector addition, often increase tensor rank, necessitating truncation that introduces potentially large errors. 
We provide direct comparisons of the proposed subspace iteration with low-rank variants of Krylov methods for symmetric matrices, i.e., Lanczos method, proposed in \cite{MPS_Lanczos,kressner2010krylov,kressner2011,Palitta2021}. These methods construct an approximate orthonormal basis of a Krylov subspace, which we show are often difficult to represent accurately using low-rank tensor formats, even when the eigenvectors themselves have low-rank structure. We further show that the truncation errors introduced while constructing the approximate Krylov subspace result in a significant loss of accuracy in the Krylov method. In contrast, polynomial-filtered subspace iteration directly uses approximate eigenvectors (Ritz vectors) as a subspace basis, avoiding the need for generating an orthonormal basis for a Krylov subspace. Through analysis and numerical experiments, we demonstrate that subspace iteration is much more robust to rank-truncation errors than the Lanczos method. To the best of our knowledge, the present paper is the first study of polynomial-filtered subspace iteration with low-rank tensor compression. Approximate power iteration using sparse matrix and vector representations has been studied in \cite{lu2020,yuan2013truncated}. 

The paper is organized as follows. In \cref{sec:tensors}, we briefly review the tensor train (TT) representation for high-dimensional tensors and basic operations involving TTs. 
In \cref{sec:chebfsi}, we introduce tensor train Chebyshev filtered subspace iteration with rank-truncation. In \cref{sec:convergence} we provide a convergence analysis that gives an upper bound on the amount of acceptable truncation error in TT subspace iteration. 
In \cref{sec:cost_and_imp} we discuss the computational cost of TT subspace iteration and different algorithms for performing rank-truncation. Finally, in \cref{sec:numerics}, we provide numerical comparisons between the proposed subspace iteration methods and Lanczos and DMRG, provide numerical demonstrations of our theoretical convergence results, and compare rank-truncation algorithms. 

\section{Low-rank tensor representation}
\label{sec:tensors}

The vector space $\mathbb{C}^{n_1\cdots n_d}$, on which \eqref{eq:sum_of_krons} acts, is naturally isomorphic to the tensor space $\mathbb{C}^{n_1 \times \cdots \times n_d}$, as tensors can be placed in one-to-one correspondence with vectors by reshaping. We refer to $d$ as the dimension and $n_k$ ($k=1,2,\ldots,d$) as the mode sizes of the tensor. The total number of entries in the tensor grows exponentially with the dimension, making the storage and computation of eigenvectors for \eqref{eq:sum_of_krons} infeasible on modern computing devices. To reduce the number of degrees of freedom, we seek eigenvectors represented in a low-rank tensor format. A low-rank tensor format factorizes a tensor $v$ into a network of low-dimensional tensors, enabling more efficient storage and computation. Several tensor formats are available, including tensor train (TT) \cite{oseledets2011tensor}, hierarchical Tucker \cite{grasedyck2010hierarchical}, canonical polyadic (CP), and Tucker \cite{kolda2009tensor}. In this paper, we focus on tensors and operators represented in the TT format, also known as the matrix product state (MPS) in the physics literature \cite{Orus2014}. The algorithms discussed in the subsequent sections can be readily adapted to other tensor formats. 
In the remainder of this section, we provide a brief overview of the TT format and arithmetic operations involving TTs.

\subsection{Tensor train format}
\label{sec:TT_format}

A tensor $v \in \mathbb{C}^{n_1 \times \cdots \times n_d}$ is said to be in the tensor train (TT) format if its entries are expressed as 
\begin{equation} \label{eq:TT_format} 
v(i_1, i_2, \ldots, i_d) = C_1(i_1) C_2(i_2) \cdots C_d(i_d),
\end{equation}
where each $C_k$ is a $r_{k-1} \times n_k \times r_{k}$ tensor, referred to as a TT-core. Here, $C_k(i_k)$ is interpreted as a $r_{k-1} \times r_k$ matrix with $r_0=r_d=1$. Thus each element of $v$ is represented as a product of $d$ matrices, hence the name matrix product state (MPS) given to this representation in physics literature. If we assume the $r_k$ are all equal to $r$ and the mode size $n_k$ are all equal to $n$, the decomposition \eqref{eq:TT_format} has $\mathcal{O}(dnr^2)$ degrees of freedom. Importantly, this number scales linearly with the number of dimensions $d$, provided the parameter $r$ remains constant. The vector $\bm r = (1,r_1,\ldots,r_{d-1},1)$ containing the minimal $r_k$ needed to represent $v$ is called the TT-rank of $v$ and the efficiency of the TT representation \eqref{eq:TT_format} relies on keeping the TT-rank as small as possible. To apply the Lanczos method or subspace iteration for computing eigenvectors of \eqref{eq:sum_of_krons} in the TT format \eqref{eq:TT_format}, we first review the essential arithmetic operations on TTs. A key challenge of iterative methods with TTs is that standard linear algebra operations, such as matrix-vector products and linear combinations, increase TT rank and hence are often followed by rank-truncation to control computational cost. Such rank-truncation is the source of inexactness in the TT algorithms to follow.

\subsubsection{Inner product of tensors in TT format}
The Euclidean inner product of two TTs $v(i_1,\ldots,i_d) = C_1(i_1)\cdots C_d(i_d)$ and $w(i_1,\ldots,i_d) = D_1(i_1)\cdots D_d(i_d)$ with compatible mode sizes can be computed in $\mathcal{O}(dnr^3)$ by performing tensor contractions, i.e., summing over shared dimensions, in an appropriate order. For more details we refer the reader to \cite[Algorithm 4]{oseledets2011tensor}. 

\subsubsection{Addition of tensors in TT format} 
\label{sec:addition}
Given TTs $v(i_1,\ldots,i_d) = C_1(i_1)\cdots C_d(i_d)$ and $w(i_1,\ldots,i_d) = D_1(i_1)\cdots D_d(i_d)$ with compatible mode sizes, the sum $v + w$ can be represented in the TT format by concatenating the cores of $v$ and $w$. The dimensions of the concatenated cores, and hence the resulting TT rank, is equal to the sum of the dimensions of $C_k$ and $D_k$. To control the TT-ranks, sums are typically followed by a rank-truncation procedure or directly approximated with a low-rank tensor, as described in \cref{sec:truncation}. 

\subsubsection{TT matrices and matrix-vector products} 

The matrix \eqref{eq:sum_of_krons} is represented in a low-rank format known as the canonical polyadic (CP) decomposition. This Kronecker product structure allows each summand to be applied efficiently to a TT. After applying all summands, the resulting sum can be approximated by a low-rank TT using the truncation procedure described in \cref{sec:truncation}. This yields an efficient method for computing approximate matrix–vector products between \eqref{eq:sum_of_krons} and a TT. 

Alternatively, any operator expressed in the CP format \eqref{eq:sum_of_krons} can be written as a TT-matrix \cite[Section 2.3]{kressner2014}, also referred to as matrix product operator (MPO) in the physics literature, as 
\begin{equation} \label{eq:tt_matrix}
    A(i_1 \cdots i_d, j_1 \cdots j_d) = 
    A_1(i_1,j_1) A_2(i_2,j_2) \cdots A_d(i_d,j_d), 
\end{equation}
where $A_k(i_k,j_k) \in \mathbb{C}^{r_{k-1} \times r_k}$. 
When both \( A \) and \( v \) are represented in the TT format, the resulting vector \( w \) can also be expressed in the TT format with the dimension of each core being the product of the corresponding cores in $A$ and $v$~\cite[Section 4.3]{oseledets2011tensor}.
Similar to the addition of TTs, to control the TT rank, matrix-vector products are typically followed by a truncation procedure or directly approximated with a low-rank TT. More details are provided in \cref{sec:truncation}.

\subsubsection{TT rank-truncation} 
\label{sec:truncation}

Common operations, such as addition and matrix-vector products between TTs, as described above, result in TTs with increased rank. To maintain low-rank vectors and hence computationally efficient algorithms, the result of these operations are typically approximated with a TT of smaller rank.
A straightforward approach to computing rank $\bm r$ TT approximations of sums and matrix-vector products involves first constructing a TT with larger ranks and then recursively applying SVD-based truncation with the 
TT-SVD algorithm \cite{oseledets2011tensor}. The user can select a maximum allowable rank $\bm r$ or a desired accuracy $\epsilon>0$. In the latter case the TT-SVD algorithm selects the rank $r_j$ by discarding all singular values smaller than $\delta = \epsilon/\sqrt{d-1}$ of a relevant matrix, so that the error introduced by truncating $w$ to rank-$\bm r$ satisfies 
\begin{equation} \label{eq:global_trunc_err}
    \|\T_{\bm r}(w) - w\|_F \leq \epsilon. 
\end{equation} 

Other approaches to approximating sums and matrix-vector products with low-rank TTs avoid constructing intermediate TTs with large ranks, offering greater computational efficiency at the expense of some accuracy. Examples include randomized algorithms \cite{al2023randomized}, projections onto a suitable TT tangent space \cite[Section 4.3]{Steinlechner2016}, and variational methods. Such approaches are discussed further in \cref{sec:cost_subspace} and \cref{sec:randtruncation}. 
%
%
In the development of the following algorithms, we use $\mathfrak{T}_{\bm r}$ to denote a rank-$\bm r$ truncation operator without specifying a particular truncation algorithm. Computational costs of different truncation algorithms are discussed in \cref{sec:cost_and_imp} and comparisons are presented in our numerical experiments in \cref{sec:improving_eff}.

\section{Chebyshev filtered subspace iteration}
\label{sec:chebfsi}
Historically, subspace iteration has been overshadowed by Krylov methods, which are often more efficient when using standard arithmetic. However, in the context of low-rank approximation, subspace iteration offers distinct advantages over Krylov methods such as those proposed in \cite{MPS_Lanczos,kressner2010krylov,kressner2011,Palitta2021} and summarized in \cref{sec:lanczos_lr}. A major challenge in low-rank Krylov methods is that orthogonal Krylov basis vectors cannot be accurately represented in low-rank tensor formats, leading to significant truncation errors that hinder convergence. In contrast, subspace iteration methods can directly target low-rank eigenvectors as the basis for the subspace
bypassing the need for a Krylov basis. After discussing low-rank inexact subspace iteration algorithms, we demonstrate in the subsequent sections that the convergence of low-rank subspace iteration is more robust to truncation errors compared to Krylov methods. 

\subsection{Power iteration}
\label{sec:power_iter}

To illustrate the rank truncated subspace iteration in the simplest setting, consider the power iteration -- a subspace iteration involving a single vector. In power iteration, the matrix $A$ is repeatedly applied to an initial vector, and under certain conditions, this process converges to the dominant eigenvector of $A$. When $A$ and the vector being iterated are represented as low-rank tensors, applying $A$ to the current iterate typically results in a vector of higher rank. In rank-truncated power iteration, this vector is truncated to a user-provided rank $\bm r$ and then normalized before proceeding to the next iteration. The main steps of power iteration with rank truncation are presented in \cref{alg:truncated_power}. In this algorithm, the operator $\mathfrak{T}_{\bm r}$ denotes an operator that truncates a TT to rank $\bm r$. The truncation rank $\bm r$ and the algorithm used to perform the truncation is set by the user. More details on choices of rank-truncation algorithms and their computational costs are discussed in \cref{sec:cost_subspace}.


\begin{algorithm}
\caption{Power iteration with rank truncation}
\label{alg:truncated_power}
\begin{algorithmic}[1]
\Require 
\Statex \hspace{1em} $A \in \mathbb{C}^{(n_1 \cdots n_d) \times (n_1 \cdots n_d)}$ in TT-matrix format \eqref{eq:tt_matrix}
\Statex \hspace{1em} $v^{(0)} \in \mathbb{C}^{n_1 \cdots n_d}$ in TT format \eqref{eq:TT_format}
\Statex \hspace{1em} $\bm r \in \mathbb{N}^{d+1}$ TT truncation rank
\Ensure
\Statex \hspace{1em} $v^{(i)} \in \mathbb{C}^{n_1 \cdots n_d}$ rank $\bm r$ TT approximating dominant eigenvector of $A$
\vspace{0.5em}
\State $i \gets 0$
\Repeat
    \State $z^{(i+1)} \gets \mathfrak{T}_{\bm r}\!\left( A v^{(i)} \right)$ \Comment{Apply $A$ and truncate}
    \State $v^{(i+1)} \gets z^{(i+1)} / \lVert z^{(i+1)} \rVert$ \Comment{Normalize}
    \State $i \gets i+1$
\Until{convergence}
\end{algorithmic}
\end{algorithm}

\subsection{Subspace iteration} 

A basic form of the subspace iteration begins with a $m$ dimensional basis. One iteration involves applying the matrix $A$ to each basis vector and then orthonormalizing the resulting vectors against each other. The matrix multiplication amplifies components of each vector in the dominant eigendirection and the orthonormalization ensures that each vector converges to a different eigenvector. Under suitable conditions, iterating these steps yields a sequence of bases that eventually span the dominant eigenspace of $A$. 
We consider a low-rank variant of subspace iteration where all vectors are approximated using low-rank TTs. To describe one iteration of the low-rank TT subspace iteration suppose we have a basis of $m$ rank-$\r$ TTs $(v_1,\ldots,v_m)$. To perform one iteration we apply the matrix $A$ to each basis vector and truncate the result 
\begin{equation}
\label{subspace_1}
z_j = \mathfrak{T}_{\bm r}\left( A v_j \right), \qquad j = 1,2,\ldots,m. 
\end{equation}
Just as with the power iteration described in \cref{sec:power_iter}, the truncation rank $\bm r$ and the algorithm used to perform the truncation is set by the user. More details on choices of rank-truncation algorithms and their computational costs are provided in \cref{sec:cost_subspace}. 
Then, we construct new basis vectors by approximately orthonoromalizing the $z_j$ in \eqref{subspace_1} using a low-rank variant of the Rayleigh-Ritz projection. 
The Rayleigh-Ritz procedure is applied to ensure that the updated basis vectors are Ritz vectors, providing approximations of the dominant eigenvectors of $A$, which we assume can be represented as low-rank TTs. Note that in subspace iteration implementations without rank truncation, it is typical to use a different orthogonalization procedure such as QR-decomposition on the vectors \eqref{subspace_1}. This however faces the same issue as Krylov subspace methods summarized in \cref{sec:krylov}. In general, the orthogonal basis vectors are not low-rank making them difficult to represent efficiently in the TT format. We summarize the main steps of low-rank subspace iteration in \cref{alg:subspace_low_rank}. In the algorithm we denote by $\texttt{RR}(A,z_1,\ldots,z_m)$ a call to the low-rank Rayleigh-Ritz procedure described hereafter. 

\begin{algorithm}
\caption{Subspace iteration with rank truncation}
\label{alg:subspace_low_rank}
\begin{algorithmic}[1]
\Require
\Statex \hspace{1em} 
$A \in \mathbb{C}^{(n_1 \cdots n_d) \times (n_1 \cdots n_d)}$ in TT-matrix format \eqref{eq:tt_matrix}
\Statex \hspace{1em} $v_1^{(0)}, \ldots, v_m^{(0)} \in \mathbb{C}^{n_1 \cdots n_d}$ system of $m$ TTs 
\Statex \hspace{1em} $\bm r \in \mathbb{N}^{d+1}$ TT truncation rank
\Ensure
\Statex \hspace{1em} 
$v_1^{(i)}, \ldots, v_m^{(i)} \in \mathbb{C}^{n_1 \cdots n_d}$ rank $\bm r$ TTs approximating dominant eigenvectors of $A$ 
\Statex \hspace{1em} 
$\lambda_1^{(i)}, \ldots, \lambda_m^{(i)} \in \mathbb{C}$ approximating $m$ dominant eigenvalues of $A$
\vspace{0.5em}
\State $i \gets 1$
\While{not converged} 
    \State $z_j^{(i)} \gets \mathfrak{T}_{\bm r}\!\left(A v_j^{(i-1)} \right)$ 
           \quad for $j = 1,\ldots,m$
    \State $(v_1^{(i)}, \ldots, v_m^{(i)}), \, (\lambda_1^{(i)}, \ldots, \lambda_m^{(i)}) 
           \gets \texttt{RR}(A, z_1^{(i)}, \ldots, z_m^{(i)})$
    \State $i \gets i+1$
\EndWhile
\end{algorithmic}
\end{algorithm}

\subsection{Rayleigh-Ritz procedure} 
\label{sec:RR}

In order to obtain low-rank approximate eigenvectors of $A$ from within the subspace $\text{span}\{z_1,\ldots,z_m\}$ where each $z_j$ is a low-rank TT, we employ a low-rank variant of the Rayleigh-Ritz procedure. To perform the Rayleigh-Ritz projection we construct the $m \times m$ matrices 
\begin{equation}
\label{projection_mats_RR}
\begin{aligned}
\quad W_{ij} = \left\langle z_i,z_j\right\rangle, \quad P_{ij} = \left\langle z_i,A z_j \right\rangle \qquad  i,j = 1,2,\ldots,m, 
\end{aligned}
\end{equation}
where the inner products are computed efficiently from the TT-cores using \cite[Algorithm 4]{oseledets2011tensor}. Applying the matrix $A$ to $z_j$ in the construction of $P$ results in a TT with increased rank. We do not truncate since the inner product that immediately follows the matrix vector product has the same computational cost as truncation. Therefore no low-rank approximation is made in the construction of the projected matrices $W$ and $P$. We then solve the generalized $m \times m$ eigenvalue problem 
\begin{equation} 
\label{gen_eig}
P \Phi = W \Phi{\Lambda} . 
\end{equation} 
The diagonal entries of ${\Lambda}$ are the approximate eigenvalues (Ritz values) and $\Phi$ contains the coefficients of the corresponding approximate eigenvectors (Ritz vectors) relative to the basis $(z_1,\ldots,z_m)$. Hence the Ritz vectors are linear combinations of the basis 
\begin{equation}
\label{eq:truncated_ritz_vecs}
\tilde{v}_j =  \mathfrak{T}_r 
\left( \sum_{i=1}^m \Phi_{ij} z_i \right), 
\qquad i = 1,2,\ldots,m. 
\end{equation}
As described in \cref{sec:addition}, adding low-rank TTs increases rank and thus each linear combination in \eqref{eq:truncated_ritz_vecs} must be approximated with a low-rank TT. Different approaches can be taken for approximating such linear combinations such as applying TT-SVD several times, approximating the sum on a tangent space of a low-rank TT manifold, or randomized algorithms. When the number of summands $m$ is large enough, we found that tangent space projections or randomized algorithms are particularly advantageous for efficiently approximating the linear combinations in \eqref{eq:truncated_ritz_vecs} with low-rank TTs. 
Finally we normalize the Ritz vectors 
$v_j = \tilde{v}_j/\|\tilde{v}_j\|_F$, which does not require any low-rank truncation. We summarize the low-rank Rayleigh-Ritz procedure in \cref{alg:TT-Rayleigh-Ritz}.

\begin{algorithm}
\caption{Rayleigh–Ritz projection with rank truncation}
\label{alg:TT-Rayleigh-Ritz}
\begin{algorithmic}[1]
\Require
\Statex \hspace{1em} $A \in \mathbb{C}^{(n_1\cdots n_d) \times (n_1\cdots n_d)}$ in TT-matrix format
\Statex \hspace{1em} $(z_1,\ldots,z_m)$ TT basis vectors for a subspace of $\mathbb{C}^{n_1\cdots n_d}$
\Statex \hspace{1em} $\bm r \in \mathbb{N}^{d+1}$ TT truncation rank
\Ensure
\Statex \hspace{1em} 
$v_1^{(i)}, \ldots, v_m^{(i)} \in \mathbb{C}^{n_1 \cdots n_d}$ rank $\bm r$ TTs approximating dominant eigenvectors of $A$ 
\Statex \hspace{1em} 
$\lambda_1^{(i)}, \ldots, \lambda_m^{(i)} \in \mathbb{C}$ approximating $m$ dominant eigenvalues of $A$
\vspace{0.5em}
\State $W_{jk} \gets \langle z_j, z_k \rangle$, \quad $j,k=1,\ldots,m$ \Comment{Gram matrix of the basis}
\State $P_{jk} \gets \langle z_j, A z_k \rangle$, \quad $j,k=1,\ldots,m$ \Comment{Projected matrix in the subspace}
\State Solve $P \Phi = \Lambda W \Phi$ \Comment{Generalized eigenproblem for Ritz values/vectors}
\State $\tilde{v}_k \gets \mathfrak{T}_{\bm r}\!\left( \sum_{j=1}^{m} \Phi_{jk} z_j \right), \quad k=1,\ldots,m$ \Comment{Truncate linear combination to rank $\bm r$}
\State $v_j \gets \tilde{v}_j / \lVert \tilde{v}_j \rVert, \quad j=1,\ldots,m$ \Comment{Normalize Ritz vectors}
\end{algorithmic}
\end{algorithm}


\subsection{Chebyshev polynomial filtering}
\label{sec:cheb_filtering}

The convergence rate of subspace iteration without truncation is governed by the eigenvalue ratios of $A$, where the $j$th eigenvector converges at a rate proportional to $\lambda_{m+1}/\lambda_j$ \cite{Saad2011}. To accelerate convergence, one can replace $A$ with a polynomial $p(A)$ of $A$, such that the eigenvalues are transformed to $p(\lambda_i)$ and the ratio $p(\lambda_{m+1})/p(\lambda_j)$ becomes smaller than $\lambda_{m+1}/\lambda_j$. Performing subspace iteration on $p(A)$ instead of $A$ thus results in faster convergence. Furthermore, the use of a polynomial transformation allows us to target eigenvalues that are not the largest in magnitude. We consider Chebyshev polynomials of the first kind \cite{Saad_2016} defined as 
\begin{equation}
c_k(t) = 
\begin{cases}
\cos\left(k \cos^{-1}(t)\right), & t \in [-1,1], \\
\cosh\left(k \cosh^{-1}(t)\right), & t \not\in [-1,1], 
\end{cases}
\end{equation}
which can be computed using the three-term recurrence 
\begin{equation}
\label{cheby_three_term}
c_0(t) = 1, \quad 
c_1(t) = t, \quad
c_{k+1}(t) = 2t c_k(t) - c_{k-1}(t). 
\end{equation}
The polynomials $c_k$ are bounded inside of the interval $[-1,1]$ and rapidly increase outside $[-1,1]$. To construct a polynomial $p$ that yields a small ratio $p(\lambda_{m+1})/p(\lambda_j)$, we compose $c_k$ with a linear transformation that maps $\lambda_1,\ldots,\lambda_m$ outside of $[-1,1]$ and  $\lambda_{m+1},\ldots,\lambda_{n_1\cdots n_d}$ inside $[-1,1]$. To construct such a linear map, suppose that all eigenvalues of $A$ lie within the real interval $[a_0,b]$ and the eigenstates of interest $\lambda_1,\ldots \lambda_m$ lie in the interval $[a_0,a]$. The linear transformation 
\begin{equation} \label{eq:linear_trans}
l(t) = \frac{t-c}{e}, \qquad c = \frac{a+b}{2}, \qquad e = \frac{b-a}{2},
\end{equation}
maps $[a,b]$ into $[-1,1]$ and $[a_0,a)$ outside of $[-1,1]$. Therefore, the polynomial $p_k = c_k \circ l$ is bounded inside $[a,b]$ and rapidly increases in $[a_0,a)$, resulting in a large eigenvalue ratio $p_k(\lambda_{m+1})/p_k(\lambda_j)$. In practice the interval $[a,b]$ is not known a priori but can be estimated. For example, we can set $b$ to the largest eigenvalue (or its approximation), and the endpoint $a$ can be updated after each Rayleigh-Ritz step to the largest Ritz value. 

To perform low-rank Chebyshev filtered power (or subspace) iteration we $Av$ with a low-rank approximation of $p_k(A)v$, where the polynomial filter degree $k$ is a user-supplied parameter. 
To compute this low-rank approximation, we utilize the three-term recurrence \eqref{cheby_three_term}, inserting rank truncations to control the intermediate tensor ranks. This yields the rank-truncated three-term recurrence 
\begin{equation} \label{eq:three_term_cheby_trunc}
q_0 = v, \quad 
q_1 = \T_{\r}\left(\frac{Av-cv}{e}\right), \quad
q_{j+1} = \T_{\r}\left(
\frac{2\T_{\r}\left(Aq_j - cq_j\right)}{e} - q_{j-1}\right), 
\end{equation}
where $q_k$ is a rank-$\r$ approximation of $p_k(A)v$. 
As mentioned above, the operator $\mathfrak{T}_{\bm r}$ denotes an operator that truncates a TT to rank-$\bm r$. The truncation rank $\bm r$ and the algorithm used to perform the truncation is set by the user. More details on choices of rank-truncation algorithms and their computational costs are discussed in \cref{sec:cost_subspace}. We summarize the Chebyshev filtered subspace iteration with rank truncation in \cref{alg:filtered_subspace_low_rank}. 

\begin{algorithm}
\caption{Chebyshev filtered subspace iteration with rank truncation}
\label{alg:filtered_subspace_low_rank}
\begin{algorithmic}[1]
\Require
\Statex \hspace{1em} 
$A \in \mathbb{C}^{(n_1 \cdots n_d) \times (n_1 \cdots n_d)}$ in TT-matrix format \eqref{eq:tt_matrix}
\Statex \hspace{1em} $v_1^{(0)}, \ldots, v_m^{(0)} \in \mathbb{C}^{n_1 \cdots n_d}$ system of $m$ TTs 
\Statex \hspace{1em} $\bm r \in \mathbb{N}^{d+1}$ TT truncation rank
\Statex \hspace{1em} $k \in \mathbb{N}$ polynomial filter degree
\Ensure
\Statex \hspace{1em} 
$v_1^{(i)}, \ldots, v_m^{(i)} \in \mathbb{C}^{n_1 \cdots n_d}$ rank $\bm r$ TTs approximating dominant eigenvectors of $A$ 
\Statex \hspace{1em} 
$\lambda_1^{(i)}, \ldots, \lambda_m^{(i)} \in \mathbb{C}$ approximating $m$ dominant eigenvalues of $A$
\vspace{0.5em}
\State $i \gets 1$
\While{not converged} 
    \State $z_j^{(i)} \gets \mathfrak{T}_{\bm r}\!\left(p_k(A) v_j^{(i-1)} \right)$ 
           \quad for $j = 1,\ldots,m$
    \State $(v_1^{(i)}, \ldots, v_m^{(i)}), \, (\lambda_1^{(i)}, \ldots, \lambda_m^{(i)}) 
           \gets \texttt{RR}(A, z_1^{(i)}, \ldots, z_m^{(i)})$
    \State $i \gets i+1$
\EndWhile
\end{algorithmic}
\end{algorithm}

\section{Upper bounds on truncation error}
\label{sec:convergence}

In this section, we analyze the convergence behavior of TT subspace iteration with rank truncation. Unlike the analysis in \cite{Saad_2016}, we derive an explicit upper bound on the inexactness at each iteration, which guarantees progress toward convergence. These results serve as bounds on the truncation error in rank-truncated power and subspace iterations that ensure convergence.

\subsection{Power iteration} \label{sec:power_iter_conv}
Before analyzing rank-truncated subspace iteration, we first examine the convergence of the simpler rank-truncated power iteration (\cref{alg:truncated_power}) for computing the dominant eigenvector of the matrix $A$ defined in \eqref{eq:sum_of_krons}. To this end, let $\{\psi_j\}$ denote an orthonormal eigenbasis of $A$, with eigenvalues ordered such that $|\lambda_1| > |\lambda_2| \geq \cdots \geq |\lambda_n|$. 

First we recall a result for classical power iteration (without rank truncation), which relies on the fact that applying the matrix $A$ to a vector $v$ amplifies the component of $v$ in the direction of the dominant eigenvector $\psi_1$ by a factor of $|\lambda_1|$, while amplifying components orthogonal to $\psi_1$ by at most $|\lambda_2|$. Assuming $|\lambda_1|>|\lambda_2|$, each application of $A$ aligns $v$ more closely with $\psi_1$. To extract the components of a vector in the dominant and non-dominant eigendirections, we define the projectors 
\begin{equation} \label{eq:projectors_1}
    \P = \psi_1 \psi_1^{\top}, \qquad \P_{\perp} = I_{n_1 \cdots n_d} - \P.
\end{equation}  
A formal proof of convergence for classical power iteration is presented in the following theorem, where $\measuredangle(v, w)$ denotes the angle between vectors $v$ and $w$. The theorem demonstrates that applying $A$ to $v$ produces a new vector with a smaller angle relative to the dominant eigenvector $\psi_1$. Consequently, repeated applications of $A$ to a given vector gradually reduce this angle. 

\begin{theorem} \label{thm:pwr_conv_no_trunc}
For any vector $v$ such that $\|\P v\| \neq 0$ it holds that 
\begin{equation} \label{eq:pwr_conv}
\tan \measuredangle(Av,\psi_1) \leq \left|\frac{\lambda_2}{\lambda_1}\right|\tan\measuredangle(v,\psi_1). 
\end{equation}
\end{theorem}

\noindent
The above theorem is a well-known result~\cite{parlett}, and we provide a proof in \cref{sec:proofpwr}. The inequality \eqref{eq:pwr_conv} stated in the theorem can be visualized by projecting the vector space $\mathbb{C}^{n_1 \cdots n_d}$ onto $\mathbb{R}^2$, where the first component represents the norm of the vector's component in the range of $\P$ (i.e., in the direction of $\psi_1$), while the second component represents the norm of the vector's component in the range of $\P_{\perp}$ (i.e., in the subspace orthogonal to $\psi_1$). We denote such projected vectors using a hat, e.g., 
\begin{equation} \label{eq:R2_proj}
\widehat{Av} = \begin{bmatrix}
\|\P Av\| \\ \|\P_{\perp} Av\|
\end{bmatrix} 
\quad \text{and} \quad 
\widehat{v} = \begin{bmatrix}
\|\P v\| \\ \|\P_{\perp} v\|
\end{bmatrix}. 
\end{equation}
\cref{fig:power_conv}(a) illustrates the convergence result of \cref{thm:pwr_convergence} for classical power iteration using these projected vectors. The vector $\widehat{Av}$ has a smaller angle with the horizontal axis than $\widehat{v}$, indicating that a single power iteration produces a vector that is more aligned with the direction of $\psi_1$. 

\subsubsection{Single power iteration with truncation} 
We now turn to the rank-truncated power iteration, outlined in \cref{alg:truncated_power}. This algorithm follows the same steps as the classical power iteration, with the key addition of a rank-truncation in line 3. We note that the following analysis is not limited to inexact power iteration arising from rank truncation and can be applied to any inexact power method. We express the rank-truncated vector produced in line 3 as 
\begin{equation} \label{eq:perturbed_power}
    z = Av + e, 
\end{equation}
where $e$ is a perturbation due to rank truncation and we suppressed the iteration index $i$ while examining a single iteration. Progress towards convergence is achieved in a single iteration if the angle between $z$ and $\psi_1$ is smaller than the angle between $v$ and $\psi_1$, i.e., $\measuredangle(z,\psi_1) < \measuredangle(v,\psi_1)$. Such progress is guaranteed when the perturbation $e$ introduced by rank-truncation is sufficiently small. 
To determine the acceptable size of $e$ that guarantees convergence we consider the vectors \eqref{eq:R2_proj} in \cref{fig:power_conv}. As long as the truncated vector remains within the green circle, the condition $\measuredangle(z,\psi_1) < \measuredangle(v,\psi_1)$ holds, ensuring progress towards convergence. Hence the radius of the green circle is the maximum allowable truncation error that ensures convergence for perturbations $e$ in arbitrary directions. This radius is the norm of the component of $\widehat{Av}$ that is orthogonal to $v$, which provides an upper bound for the truncation error 
\begin{equation} \label{eq:err_bnd1}
\|e\| < \left\|\left(I_2 - \widehat{v}\widehat{v}^{\top}\right)\widehat{Av}\right\|, 
\end{equation}
where $\|v\|=1$ and $I_2$ denotes the $2 \times 2$ identity matrix. This upper bound on the truncation error ensures convergence even in the worst case scenario, where the projected truncation vector $\widehat{e}$ is orthogonal to $\widehat{v}$. However, such a scenario does not occur in our numerical experiments. Instead, we observe that, in many cases, truncation can actually accelerate convergence. This is illustrated in \cref{fig:power_conv}(a), where $\measuredangle(z,\psi_1)<\measuredangle(Av,\psi_1)$, and demonstrated numerically in \cref{fig:laplace_2d_conv_iters}. 

\begin{figure}[tbhp]
\centering
\includegraphics[scale=0.24]{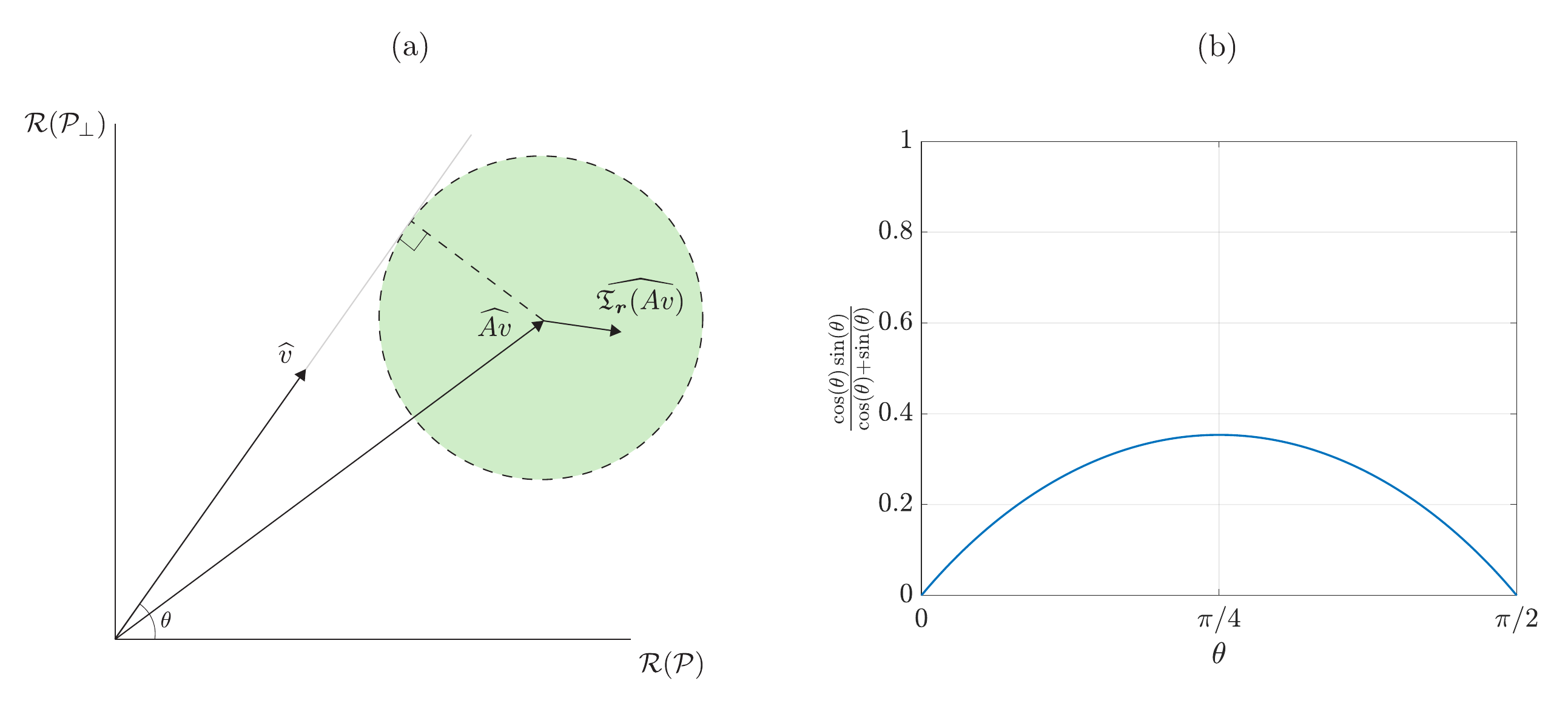} 
\caption{(a) Illustration of one iteration of truncated power iteration. The radius of the dashed ball is the upper bound on the size of absolute truncation error $\|e\|$ that guarantees progress is made towards convergence in a single iteration. 
(b) Upper bound (re-scaled by eigenvalue ratio) on truncation error sufficient for truncated power iteration convergence in \cref{thm:pwr_convergence}. 
}
\label{fig:power_conv}
\end{figure}

Next, we derive a more restrictive bound on the truncation error $e$ in terms of the eigenvalues of $A$. Since the truncation is applied to a $Av$, we bound norm of the truncation error relative to the largest singular value $|\lambda_1|$. 

\begin{theorem} \label{thm:pwr_convergence} 
Let $\theta = \measuredangle(v,\psi_1)$ denote the angle between $v$ and $\psi_1$. If the relative truncation error is bounded as 
\begin{equation} \label{eq:trunc_error_bound_angle}
    \frac{\|e\|}{|\lambda_1|} < \left(1 - \left|\frac{\lambda_2}{\lambda_1}\right|\right)
    \frac{\cos(\theta) \sin(\theta)}{\cos(\theta) + \sin(\theta)}, 
\end{equation}
then $\measuredangle(Av+e,\psi_1) < \measuredangle(v,\psi_1)$, i.e., a single inexact power iteration makes progress towards convergence. 
\end{theorem}
\begin{proof}
Substituting $\cos(\theta) = \|\P v\|$ and $\sin(\theta)=\|\P_{\perp}v\|$ we rewrite the inequality \eqref{eq:trunc_error_bound_angle} as 
\begin{equation}
\|e\| < \left(|\lambda_1| - |\lambda_2| \right)\frac{\|\P v\| \cdot \|\P_{\perp} v\|}{\|\P v\| + \|\P_{\perp} v\|}. 
\end{equation}
Multiplying by $\|\P v\| + \|\P_{\perp} v\|$ we have 
$$
\|\P v\|\cdot \|e\| + \|\P_{\perp} v\| \cdot \|e\| < \left(|\lambda_1| - |\lambda_2| \right) \|\P v\| \cdot \|\P_{\perp} v\|. 
$$
Dividing by $\|\P_{\perp} v\|$ we obtain 
$$
\frac{\|\P v\|}{\|\P_{\perp} v\|} \cdot \|e\| + \|e\| < |\lambda_1| \|\P v\| - |\lambda_2|\|\P v\|. 
$$
Rearranging terms the preceding inequality can be written as 
$$
|\lambda_1| \|\P v\| - \|e\| > \frac{\|\P v\|}{\|\P_{\perp} v\|} \left( |\lambda_2| \|\P_{\perp} v\| + \|e\|\right). 
$$
Dividing by $\left( |\lambda_2| \|\P_{\perp} v\| + \|e\|\right)$ we have 
\begin{equation} \label{eq:bnd1}
\frac{|\lambda_1| \|\P v\| - \|e\|}{|\lambda_2| \|\P_{\perp} v\| + \|e\|} > \frac{\|\P v\|}{\|\P_{\perp} v\|}. 
\end{equation}
It follows from 
\begin{equation} 
\left\| \P_{\perp} Av \right\| \leq  |\lambda_2|\left\|\P_{\perp} v\right\|
\quad \text{and} \quad 
\left\| \P Av \right\| = |\lambda_1|\left\|\P v\right\|. 
\end{equation} 
that the left hand side of \eqref{eq:bnd1} can be bounded from above by 
\begin{equation} \label{eq:ineq1}
\frac{\| \P Av\| - \|e\|}{ \left\|\P_{\perp} \left(Av + e\right) \right\|} 
\leq \frac{\| \P Av\| - \|\P e\|}{ \left\|\P_{\perp} \left(Av + e\right) \right\|} 
\leq \frac{\| \P Av + \P e\|}{ \left\|\P_{\perp} \left(Av + e\right) \right\|} \\
\leq \frac{\| \P \left(Av + e\right)\|}{ \left\|\P_{\perp} \left(Av + e\right) \right\|} 
\end{equation}
Consequently, we have 
\begin{equation}
\frac{\| \P \left(Av + e\right)\|}{ \left\|\P_{\perp} \left(Av + e\right) \right\|}  \geq \frac{|\lambda_1| \|\P v\| - \|e\|}{|\lambda_2| \|\P_{\perp} v\| + \|e\|} > \frac{\|\P v\|}{\|\P_{\perp} v\|}, 
\end{equation} 
completing the proof. 
\end{proof}

\noindent
The upper bound given in \eqref{eq:trunc_error_bound_angle} depends on the eigenvalue ratio $|\lambda_2/\lambda_1|$ and the angle $\theta = \measuredangle(v,\psi_1)$, which quantifies the error of the current iterate $v$ as an approximation of the dominant eigenvector $\psi_1$. In \cref{fig:power_conv}(b), we plot the upper bound as a function of the angle $\theta$ without the scaling factor due to the eigenvalue ratio. When $\theta$ is near zero, i.e., the angle between $v$ and $\psi_1$ is small, smaller truncation errors are required to maintain convergence. At this stage, significant truncation error could dominate the approximation error and hinder further convergence. When $\theta$ is close to $\pi/2$ the component of $v$ in the $\psi_1$ direction is small and a small truncation error is required to ensure that this component is able to grow. When $\theta$ is near $\pi/4$ the component of $v$ in the $\psi_1$ direction is similar to the component of $v$ in the non-dominant eigendirections. During this stage of convergence the largest truncation errors are permitted. 

Now let us assume that the dominant eigenvector $\psi_1$ is represented exactly as a rank-$\bm r$ TT. Denote the angle between the $i$th truncated power iterate $v^{(i)}$ and the dominant eigenvector $\psi_1$ by 
\begin{equation}
\theta^{(i)} = \measuredangle\left(v^{(i)},\psi_1\right). 
\end{equation}
Since we have normalized $v^{(i)}$ in line 4 of \cref{alg:truncated_power}, we have \begin{equation}
\sin\left(\theta^{(i)}\right) = \left\|\P_{\perp} v^{(i)}\right\|, 
\end{equation}
where $\P_{\perp}$, defined in \eqref{eq:projectors_1}, is the orthogonal projection onto the space orthogonal to $\psi_1$. Given that $\psi_1$ has rank-$\r$ by assumption, the error of the best rank-$\r$ approximation of $v^{(i)}$ is bounded as 
\begin{equation}
\left\|\T_{\r}\left(v^{(i)}\right) - v^{(i)}\right\| \leq \left\|\P_{\perp} v^{(i)}\right\| = \sin\left(\theta^{(i)}\right). 
\end{equation} 
Thus, as $v^{(i)}$ aligns more closely with $\psi_1$, the rank-$\r$ truncation error naturally decreases to zero.

\subsubsection{Several power iterations with truncation}
\label{sec:several_iters}

Next we discuss the effect of rank truncation over several power iterations. After $p$ iterations, a simple recursive argument yields
\begin{equation} \label{eq:several_iterations}
    v^{(p)} = A^p v^{(0)} + \sum_{i=1}^p A^{p-i} e^{(i)}, 
\end{equation}
where $e^{(i)}$ denotes a perturbation due to rank truncation at iteration $i$. In \eqref{eq:several_iterations} each error vector $e^{(i)}$ is multiplied by $A^{p-i}$ so that as $p$ increases, truncation error from early iterations are amplified by large powers of $A$. Consequently, these errors become more aligned with $\psi_1$, mitigating their impact on convergence. 
In contrast, low-rank Krylov methods (see \cref{sec:krylov}) are more sensitive to large truncation errors in the early iterations. Such errors can cause significant deviations from the true Krylov subspace, often resulting in stagnation and poor convergence. This distinction highlights a key advantage of rank-truncated subspace iteration over rank truncated Krylov methods. 

\subsubsection{Polynomial accelerated power iteration} 
\label{sec:poly_filter} 

For rank-truncated polynomial accelerated power iteration, the approximate matrix multiplication in line 3 of \cref{alg:truncated_power} is replaced with a low-rank approximation of $p_k(A)v$ computed using the truncated three-term recurrence \eqref{eq:three_term_cheby_trunc}. We can directly apply results from the preceding section to Chebyshev filtered power iteration by writing the result of the three-term recurrence with truncation \eqref{eq:three_term_cheby_trunc} as $z = p_k(A)v + e$ where $e$ is the low-rank approximation error accumulated during the truncated three-term recurrence. The bound in \eqref{eq:trunc_error_bound_angle} for the truncation error can then be directly applied to the polynomial accelerated power iteration. 
Observe that increasing the polynomial degree $k$ increases both the truncation error and the tolerance for this error. Indeed, the number of truncations required to compute a low-rank approximation of $p_k(A)v$ using the truncated three-term recurrence \eqref{eq:three_term_cheby_trunc} increases with $k$, which often leads to larger truncation error $e$. At the same time, increasing $k$ causes $p_k$ to grow faster, resulting in a smaller eigenvalue ratio $p_k(\lambda_2)/p_k(\lambda_1)$, hence a larger upper bound for the truncation error in \eqref{eq:trunc_error_bound_angle}. 

\subsection{Subspace iteration}

We now extend the analysis of rank-truncated power iteration from \cref{sec:power_iter_conv} to subspace iteration. Specifically, we derive an upper bound on the truncation error that ensures each inexact subspace iteration makes progresses towards convergence. To this end, we arrange the $m$ dominant eigenvectors of $A$ as the columns of 
$\Psi_{\leq m} \in \mathbb{C}^{(n_1 \cdots n_d) \times m}$, the remaining eigenvectors of $A$ as the columns of $\Psi_{>m}\in \mathbb{C}^{(n_1 \cdots n_d) \times (n_1\cdots n_d - m)}$, and the $m$ vectors from a single subspace iteration as the columns of $V \in \mathbb{C}^{(n_1 \cdots n_d) \times m}$. As before, we omit the iteration index $i$ to simplify notation while analyzing a single iteration. 
To measure convergence, we consider the largest principal angle between the ranges of $V$ and $\Psi_{\leq m}$, denoted by $\Theta(\R(V),\R(\Psi_{\leq m}))$. The tangent of this principal angle is given by \cite[Theorem 5.1]{knyazev2012principal} 
\begin{equation} \label{eq:tan_princ_angle}
\tan \Theta(\R(V),\R(\Psi_{\leq m})) = \sigma_1(T_V),
\end{equation}
where 
\begin{equation} \label{eq:Tv_def}
T_V = \Psi_{>m}^{\top} V \left(\Psi_{\leq m}^{\top} V\right)^{-1}, 
\end{equation}
and $\sigma_1(\bullet)$ returns the largest singular value of its argument. Before addressing the effect of rank truncation, we first establish convergence for classical subspace iteration. Specifically, we show that one iteration without rank truncation reduces the tangent of the largest principal angle between the approximate subspace and the target subspace $\R(\Psi_{\leq m})$ by at least a factor of $|\lambda_{m+1}/\lambda_m|$. This generalizes the convergence result of \cref{thm:pwr_conv_no_trunc} from power iteration to subspace iteration. We assume that the projection of the current subspace $V$ onto the target subspace $\Psi_{\leq m}$ has dimension $m$, i.e., $\Psi_{\leq m}^{\top} V$ is invertible, a standard assumption in the analysis of subspace iteration (see, e.g., \cite[Theorem 5.2]{Saad2011}). 
\begin{theorem} \label{thm:sub_conv}
If $\Psi_{\leq m}^{\top} V$ is invertible then 
\begin{equation}
    \tan\Theta(\R(AV),\R(\Psi_{\leq m})) \leq \left|\frac{\lambda_{m+1}}{\lambda_m}\right| \tan\Theta(\R(V),\R(\Psi_{\leq m})). 
\end{equation}
\end{theorem}
The above theorem is a well-known result, see, e.g., \cite{Saad_2016}, we also provide a proof in \cref{sec:proofsub} that is consistent with the notation in the present paper.

\subsubsection{Subspace iteration with truncation}

Next we analyze subspace iteration with rank truncation, described in \cref{alg:subspace_low_rank}. This algorithm involves two sources of truncation error: the matrix-vector product in line 3 and the approximate Rayleigh-Ritz procedure in line 4. 
To evaluate the impact of these truncations on convergence, we express the rank-truncated vectors from line 3 as $z_j = Av_j + e_j^{\texttt{mv}}$ for $j=1,\ldots,m$, or equivalently in matrix form as $Z = AV + E^{\texttt{mv}}$, 
where the $j$th column of $E^{\texttt{mv}}$ is the truncation error $e_j^{\texttt{mv}}$ from the approximate matrix-vector product. In line 4, the rank-truncated Rayleigh-Ritz procedure (\cref{alg:TT-Rayleigh-Ritz}) is performed, which yields the updated basis vectors $v'_k = \sum_{j=1}^{m} \Phi_{jk}z_j + e^{\texttt{s}}_k$, where $e^{\texttt{s}}_k$ is the truncation error due to summation (here we have omitted normalization in line 5 of the Rayleigh-Ritz \cref{alg:TT-Rayleigh-Ritz} for simplicity as it does not introduce truncation error or affect the convergence). The updated basis vectors can also be expressed in matrix form as 
\begin{equation}
    V' =  Z \Phi + E^{\texttt{s}} 
    = AV\Phi + E^{\texttt{mv}}\Phi + E^{\texttt{s}} \\
    = AV\Phi + E,
\end{equation}
where $E=E^{\texttt{mv}}\Phi + E^{\texttt{s}}$ accounts for the total truncation error due to truncated matrix-vector products and the truncated Rayleigh-Ritz procedure. The following result is an upper bound on the truncation error matrix $E$ that guarantees one iteration of rank-truncated subspace iteration makes progress towards convergence. Note that the following result is not limited to rank-truncated subspace iteration. It is valid for any inexact subspace iteration. 

\begin{theorem} \label{thm:sub_conv_t}
    If the truncation error matrix $E$ is bounded as 
    \begin{equation} \label{eq:sub_err_bound}
        \frac{\|E\|_2}{|\lambda_m|} < \left(\frac{1}{\kappa_2(\Psi_{\leq m}^{\top} V \Phi)} - \left|\frac{\lambda_{m+1}}{\lambda_{m}}\right|\right) \frac{\norm{\Plm^\top V \Phi}_2 \cdot \norm{\Pgm^\top V \Phi}_2}{\norm{V \Phi}_2},
    \end{equation}
    where $\kappa_2(\bullet)$ returns the 2-norm condition number of its argument, 
    then 
    \begin{equation} \label{eq:angle_bnd}
    \Theta(\R(V'),\R(\Psi_{\leq m})) < \Theta(\R(V),\R(\Psi_{\leq m}))
    \end{equation}
    i.e., one iteration of inexact subspace iteration makes progress towards convergence. 
\end{theorem}

\begin{proof}
The left-hand side of the inequality \eqref{eq:angle_bnd} is the largest singular value of the matrix $T_{V'}$ which we write as 
\begin{equation} \label{eq:perturbed-T-matrix}
\begin{aligned}
T_{V'}
&= 
\Pgm^\top(AV\Phi+E) \left[\Plm^\top\left(AV\Phi+E\right)\right]^{-1} 
\\
&= 
\left(\Lgm V_{>m} \Phi + E_{>m}\right) \left(\Llm V_{\leq m} \Phi + E_{\leq m}\right)^{-1},
\end{aligned}
\end{equation}
where $V_{\leq m} := \Psi_{\leq m}^{\top} V$,  $V_{>m} := \Psi_{>m}^{\top} V$,
$E_{\leq m} = \Plm^\top E$, and $E_{>m} = \Pgm^\top E$. 

We are to show that inequality \eqref{eq:sub_err_bound} implies $\sigma_1(T_{V'}) < \sigma_1(T_V)$. We begin by bounding $\sigma_1(T_{V'})$ from above using the inequality $\sigma_1(XY) \leq \sigma_1(X) \sigma_1(Y)$ for any appropriately sized matrices $X$ and $Y$ \cite{wang1992svBounds}. Applying such inequality to \eqref{eq:perturbed-T-matrix} we have 
\begin{equation} \label{eq:perturbed-error-bnd1}
\begin{aligned}
\sigma_1(T_{V'})
&\le
\sigma_1 (\Lgm V_{>m} \Phi + E_{>m}) \cdot \sigma_1 \left( (\Llm V_{\leq m} \Phi + E_{\leq m})^{-1}  \right)
\\
&=
\frac{\sigma_1 \left(\Lgm V_{>m} \Phi + E_{>m}\right)}{\sigma_m \left( \Llm V_{\leq m} \Phi + E_{\leq m} \right)},
\end{aligned}
\end{equation}
where the equality in the second line follows from the identity $\sigma_1(X^{-1}) = 1 / \sigma_m(X)$ for any invertible $X \in \mathbb{C}^{m \times m}$. 
To proceed, we bound the numerator of \eqref{eq:perturbed-error-bnd1} from above 
\begin{equation}
\begin{aligned}
\sigma_1\left(\Lgm V_{>m} \Phi + E_{>m}\right) 
&\leq \sigma_1\left(\Lgm V_{>m} \Phi\right) + \norm{E_{>m}}_2 \\
&\leq 
\norm{\Lambda_{>m}}_2 \norm{V_{>m} \Phi}_2 + \norm{E_{>m}}_2 \\
&\leq |\lambda_{m+1}| \cdot \norm{V_{>m} \Phi}_2 + \norm{E_{>m}}_2 
\end{aligned}
\end{equation}
where the first inequality is due to \cite[Theorem 1]{stewart1998perturbation}. 
We also bound the denominator of \eqref{eq:perturbed-error-bnd1} from below 
\begin{equation} \label{eq:denom_bound}
\begin{aligned}
\sigma_m\left(\Llm V_{\leq m} \Phi + E_{\leq m}\right) 
&\geq \sigma_m\left(\Llm V_{\leq m} \Phi\right) - \norm{E_{\leq m}}_2 \\
&\geq \sigma_m(\Llm)\ \sigma_m(V_{\leq m} \Phi)
 - \norm{E_{\leq m}}_2 \\
&\geq |\lambda_m| \cdot \sigma_m(V_{\leq m} \Phi) - \norm{E_{\leq m}}_2
\end{aligned}
\end{equation} 
where the first inequality is once again due to \cite[Theorem 1]{stewart1998perturbation}. Combining inequalities \eqref{eq:perturbed-error-bnd1}-\eqref{eq:denom_bound} yields 
\begin{equation} \label{bnd:sigma-Tvp-with-SVs}
\sigma_1(T_{V'}) \leq \frac{|\lambda_{m+1}| \cdot \norm{V_{>m} \Phi}_2 + \norm{E_{>m}}_2}{|\lambda_m| \cdot \sigma_m(V_{\leq m} \Phi) - \norm{E_{\leq m}}_2}. 
\end{equation}
We now bound $\sigma_1(T_V)$ from below 
\begin{equation} \label{eq:sTv_bnd}
\begin{aligned}
\sigma_1(T_V) &= \sigma_1\left((V_{>m}\Phi) (V_{\leq m}\Phi)^{-1}\right) \\ 
&\geq \sigma_1(V_{>m} \Phi) \cdot \sigma_m\left((V_{\leq m}\Phi)^{-1}\right) 
= \frac{\sigma_1(V_{>m} \Phi)}{\sigma_1(V_{\leq m}\Phi)} 
= \frac{\norm{V_{>m} \Phi}_2}{\norm{V_{\leq m} \Phi}_2}, 
\end{aligned}
\end{equation}
where the inequality is due to \cite{wang1992svBounds}. 
Using inequalities \eqref{bnd:sigma-Tvp-with-SVs} and \eqref{eq:sTv_bnd}, for $\sigma_1(T_{V'}) < \sigma_1(T_V)$ it is sufficient that 
\begin{equation} 
\frac{|\lambda_{m+1}| \cdot \norm{V_{>m} \Phi}_2 + \norm{E_{>m}}_2}{|\lambda_m| \cdot \sigma_m(V_{\leq m} \Phi) - \norm{E_{\leq m}}_2}
<
\frac{\norm{V_{>m} \Phi}_2}{\norm{V_{\leq m} \Phi}_2}.
\end{equation}
Multiplying the preceding inequality by $\norm{V_{\leq m} \Phi}_2(|\lambda_m| \sigma_m(V_{\leq m} \Phi) - \norm{E_{\leq m}}_2)$ we obtain 
\begin{equation}
\begin{aligned}
&\norm{V_{\leq m} \Phi}_2 \left(|\lambda_{m+1}| \norm{V_{>m} \Phi}_2 + \norm{E_{> m}}_2\right)  < \norm{V_{>m} \Phi}_2 \left(|\lambda_m| \sigma_m(V_{\leq m} \Phi) - \norm{E_{\leq m}}_2\right),
\end{aligned}
\end{equation}
and rearranging terms yields 
\begin{equation} \label{eq:suff_sub_conv}
\begin{aligned}
& \norm{V_{\leq m} \Phi}_2 \norm{E_{> m}}_2 + \norm{V_{>m} \Phi}_2 \ \norm{E_{\leq m}}_2 \\
&\qquad < \left(|\lambda_{m}| \sigma_m(V_{\leq m} \Phi) - |\lambda_{m+1}|\norm{V_{\leq m} \Phi}_2 \right) \norm{V_{>m} \Phi}_2. 
\end{aligned}
\end{equation}
The left-hand side of the preceding inequality is the inner product 
$$
\norm{V_{\leq m} \Phi}_2 \ \norm{E_{>m}}_2 + \norm{V_{>m} \Phi}_2 \ \norm{E_{\leq m}}_2 
= 
\left\langle \begin{bmatrix}
\norm{V_{\leq m} \Phi}_2 \\ \norm{V_{>m} \Phi}_2
    \end{bmatrix}
,
    \begin{bmatrix}
    \norm{E_{>m}}_2 \\ \norm{E_{\leq m}}_2
    \end{bmatrix} 
\right\rangle,
$$
which we bound using the Cauchy-Schwarz inequality
\begin{equation} \label{eq:Cauchy-Schwarz} 
\begin{aligned}
\norm{V_{\leq m} \Phi}_2 \norm{E_{>m}}_2 + \norm{V_{>m} \Phi}_2 \ \norm{E_{\leq m}}_2  &\leq
\norm{ 
    \begin{bmatrix}
    \norm{V_{\leq m} \Phi}_2 \\ \norm{V_{>m}\Phi}_2
    \end{bmatrix}
}_2 
\norm{ 
    \begin{bmatrix}
    \norm{E_{>m}}_2 \\ \norm{E_{\leq m}}_2
    \end{bmatrix}
}_2 \\ 
&= \norm{V \Phi}_2 \norm{E}_2. 
\end{aligned}    
\end{equation}
Combining \eqref{eq:suff_sub_conv} and \eqref{eq:Cauchy-Schwarz}, a sufficient condition for $\sigma_1(T_{V'})<\sigma_1(T_V)$ is 
$$
\norm{V \Phi}_2\norm{E}_2 < \left(|\lambda_m| \sigma_m(V_{\leq m} \Phi) - |\lambda_{m+1}| \norm{V_{\leq m} \Phi}_2\right) \norm{V_{>m} \Phi}_2  
$$
or equivalently, dividing by $\norm{V \Phi}_2$ and factoring $\norm{V_{\leq m} \Phi}_2$ from the right-hand side 
\begin{equation}
\norm{E}_2 < \left(\frac{|\lambda_m|}{\kappa_2(V_{\leq m} \Phi)} - |\lambda_{m+1}|\right) \frac{\norm{V_{\leq m} \Phi}_2 \norm{V_{>m} \Phi}_2}{\norm{V \Phi}_2}
\label{bnd:subspace-single-step}
\end{equation}
where $\kappa_2(V_{\leq m} \Phi) = \sigma_1(V_{\leq m} \Phi) / \sigma_m(V_{\leq m} \Phi)$. Finally dividing the preceding inequality by $|\lambda_{m}|$ we arrive at \eqref{eq:sub_err_bound}. 
\end{proof}

\vs
\noindent
Note that in \eqref{eq:sub_err_bound} we bound the truncation error matrix relative to the magnitude of the $m$-th eigenvalue of $A$. This relative error is appropriate because the error matrix arises from truncating vectors that have been multiplied by $A$. The upper bound depends on the eigenvalue ratio $|\lambda_{m+1} / \lambda_m|$, with a larger eigenvalue gap allowing for greater truncation error. 
The bound also depends on the condition number of the matrix $\Psi_{\leq m}^{\top} V \in \mathbb{C}^{m \times m}$, which represents the projection of the approximate subspace $V$ onto the target subspace $\Psi_{\leq m}$. Near convergence, this projection matrix approaches the identity, and consequently, its condition number approaches $1$. 
Additionally, the bound includes $\|\Psi_{>m}^{\top} V\|_2$ and $\|\Psi_{\leq m}^{\top} V\|_2$. When the algorithm is far from convergence, the latter is small. Conversely, when close to convergence, the former becomes small. 

If the eigenvectors of interest $\psi_1, \ldots, \psi_m$ have rank less than or equal to $\bm{r}$, the truncation error $E^{\texttt{mv}}$ resulting from the matrix-vector products in line 3 of the truncated subspace iteration (\cref{alg:subspace_low_rank}) naturally decreases as the method converges. Similarly, the rank truncation in line 4 of the low-rank Rayleigh-Ritz procedure (\cref{alg:TT-Rayleigh-Ritz}) introduces error that also decreases as the subspace iteration converges. This is because $v_j$'s converge to $\psi_j$'s, for $j=1,2,...,m$, which are naturally orthogonal for symmetric $A$. As a result, the linear combinations in the Rayleigh-Ritz step involve fewer terms of significant magnitude, thereby reducing truncation error. 

\section{Computational cost and implementation}
\label{sec:cost_and_imp}
In this section, we estimate the computational cost and discuss several practical issues that are important for efficient implementations of subspace iteration with rank truncation. 
\subsection{Computing TT matrix-vector products}
\label{sec:cost_subspace} 
The computational cost of the low-rank subspace iteration (\cref{alg:subspace_low_rank}) depends primarily on the rank of $A$ as a low-rank operator and the ranks used to represent vectors during each iteration. Assume $A$ is represented as a TT-matrix with ranks uniformly equal to $r_A$, while the Ritz vectors $v_j^{(i)}$ and all intermediate vectors, such as $z_j^{(i)}$, are represented with ranks uniformly equal to $r$. 
Another key factor influencing computational cost is the method used to approximate matrix-vector products. When matrix-vector products are approximated using TT-SVD, the product $Av_j^{(i)}$ has a rank of $r_A \cdot r$. Consequently, performing each truncated matrix-vector product in line 3 of \cref{alg:subspace_low_rank} incurs a cost of $\mathcal{O}(dnr_A^3r^3)$, assuming the mode sizes are uniformly equal to $n$. On the other hand, if the matrix-vector product is approximated by projecting onto the tangent space of the low-rank TT manifold at $v_j^{(i-1)}$, the computational cost is reduced to $\mathcal{O}(dnr^3)$. This approach is significantly more efficient, paricularly when $A$ has moderate to large ranks, as it avoids the higher costs associated with applying TT-SVD to high-rank TTs.

\subsection{Computing linear combination of low-rank TTs} 
\label{sec:randtruncation}
In the Rayleigh-Ritz step (\cref{alg:TT-Rayleigh-Ritz}), approximating each linear combination in line 4 using the TT-SVD approach involves applying a truncation after each addition, resulting in $m-1$ truncations. Each truncation has a computational cost of $\mathcal{O}(dnr^3)$ and the ordering of the summands can affect the final result. Alternatively, approximating the entire sum directly on the tangent space of the TT manifold at $z_k$ requires a single orthogonal projection, which costs $\mathcal{O}(dnr^3)$. For a large number of summands $m$, the tangent space projection becomes significantly more efficient. 

Another possibility to accelerate low-rank approximations of linear combinations of several TTs is to use randomized algorithms \cite{Daas2021}. These algorithms take advantage of the block structured cores of TT sums, and are particularly efficient when truncating large linear combinations of tensors in the TT format. Randomized rounding can be readily employed to speed up the Rayleigh Ritz procedure during TT subspace iteration, and further acceleration can be achieved by reusing sketches for each linear combination. 

\subsection{Eigenvector locking} 
\label{sec:locking}
A common strategy to enhance the efficiency of subspace iteration, referred to as locking, is to halt updates for a Ritz vector once it has converged. In low-rank subspace iteration, constructing the Ritz vectors \eqref{eq:truncated_ritz_vecs} at each iteration involves a linear combination of $m$ TTs. This process can be computationally expensive due to the rank increase that results from summing multiple TTs. To mitigate some of this cost, the locking strategy can be further extended. 
It is important to note that not all summands contribute equally to the resulting Ritz vector. As the Ritz vectors converge, some coefficients $\Phi_{ij}$ may fall below the tolerance used for low-rank truncations. In such cases, excluding vectors with small coefficients from the sum can significantly reduce computational costs by limiting unnecessary rank growth and truncation operations. By omitting these negligible terms, low-rank subspace subspace iteration can be more efficient without compromising accuracy or convergence.

\section{Numerical results}
\label{sec:numerics} 
In this section, we present several numerical examples to demonstrate the efficiency and effectiveness of the rank truncated subspace iteration and compare it with both the rank truncated Lanczos algorithm and the DMRG algorithm for computing several algebraically smallest eigenvalues. To initialize the linear transformation defining the Chebyshev polynomial \eqref{eq:linear_trans} we estimate the largest eigenvalue using a few rank-truncated Lanczos iterations with safeguards as outlined in \cite{spec_bound} and set $b$ equal to such approximation. We initialize $a=-1$ and update $a$ to be the largest Ritz value after each subspace iteration. 

\subsection{Test problems}
We test the algorithms discussed above using three problems. 
\begin{enumerate}
\item \textit{Heisenberg Hamiltonians}. 
We consider Heisenberg Hamiltonians with an external magnetic field of the form 
\begin{equation}
\label{Heis}
H = \sum_{j=1}^L -J \left( \sigma_j^x \sigma_{j+1}^x + \sigma_j^y \sigma_{j+1}^y + \sigma_j^z \sigma_{j+1}^z\right) - h  \sigma_j^z,
\end{equation} 
where $\sigma_j^{a} = I^{\otimes j-1} \otimes \sigma^{a} \otimes I^{\otimes L-j}$ for $a \in \{x,y,z\}$ denotes spin-operators at site $j$ of a $1$-dimensional lattice of length $L$. 
We consider the spin-$1/2$ model, where $I$ denotes the $2 \times 2$ identity matrix, and $\sigma^a$ denotes $2 \times 2$ Pauli matrices. When $J,h=1$ the eigenvectors associated with the smallest algebraic eigenvalues (the ground state and the first few excited states) exhibit a low TT rank. 

We also examine the spin-1 model with periodic boundary conditions, where $I$ is the $3 \times 3$ identity matrix, and $\sigma^a$ corresponds to the $3 \times 3$ Pauli matrices. In this case, the ground state and excited states generally exhibit a higher TT rank. Moreover, the long-range interactions introduced by the periodic boundary conditions are known to cause convergence difficulties for the DMRG method when computing these eigenvectors. 

\item The \textit{Laplacian operator} in a $d$-dimensional space is defined by
\begin{equation} \label{eq:laplacian-operator}
    \Delta = -\left(D \otimes I \otimes \cdots \otimes I + \cdots + 
             I \otimes \cdots \otimes I \otimes D\right), 
\end{equation}
where $I$ is the identity matrix of size $n \times n$ and $D=\text{tridiag}(1,-2,1)$ is the $n \times n$ second derivative finite difference discretization. To compare truncated subspace iteration with Lanczos method, we considered $d=3$ and $n = 16$. To illustrate the analysis of truncated power/subspace iteration in \cref{sec:convergence}, we considered $d=2$, $n=32$. In both cases the eigenvectors corresponding to the smallest eigenvalues can be represented exactly using low-rank TTs. 
%
\item \textit{A Hamiltonian with a H\'enon-Heiles potential} can be written as
\begin{equation}
    \label{henon_heiles}
    H = -\frac{1}{2} \Delta 
        + \frac{1}{2} \sum_{k=1}^d q_k^2 
        + \mu \sum_{k=1}^{d-1}\left( 
        q_k^2 q_{k+1} - \frac{1}{3} q_{k+1}^3
        \right), 
\end{equation}
where $\mu$ controls the contribution of the anharmonic term. Following \cite{dolgov2014} we set $\mu=0.111803$ and discretize $H$ using spectral collocation with a tensor product grid based on the zeros of the $n$th Hermite polynomial. We set $d=5$ and $n=28$. The low-lying eigenvectors of the H\'enon-Heiles Hamiltonian \eqref{henon_heiles} exhibit higher ranks compared to the low-lying eigenvectors of the spin-1/2 Heisenberg Hamiltonian and the Laplacian operator. This makes the H\'enon-Heiles system well-suited for assessing the performance of the truncated subspace iteration when the truncation rank is smaller than the rank of the eigenvectors being approximated. 

\end{enumerate}

Each of the test problem Hamiltonians are written above in the canonical polyadic (CP) format \eqref{eq:sum_of_krons}. In order to apply the algorithms described above, each Hamiltonian is first converted from CP format to TT-matrix format. 

\subsection{Comparing subspace iteration and Lanczos method} 

\begin{figure}[tbhp]
\centering
\includegraphics[scale=0.20]{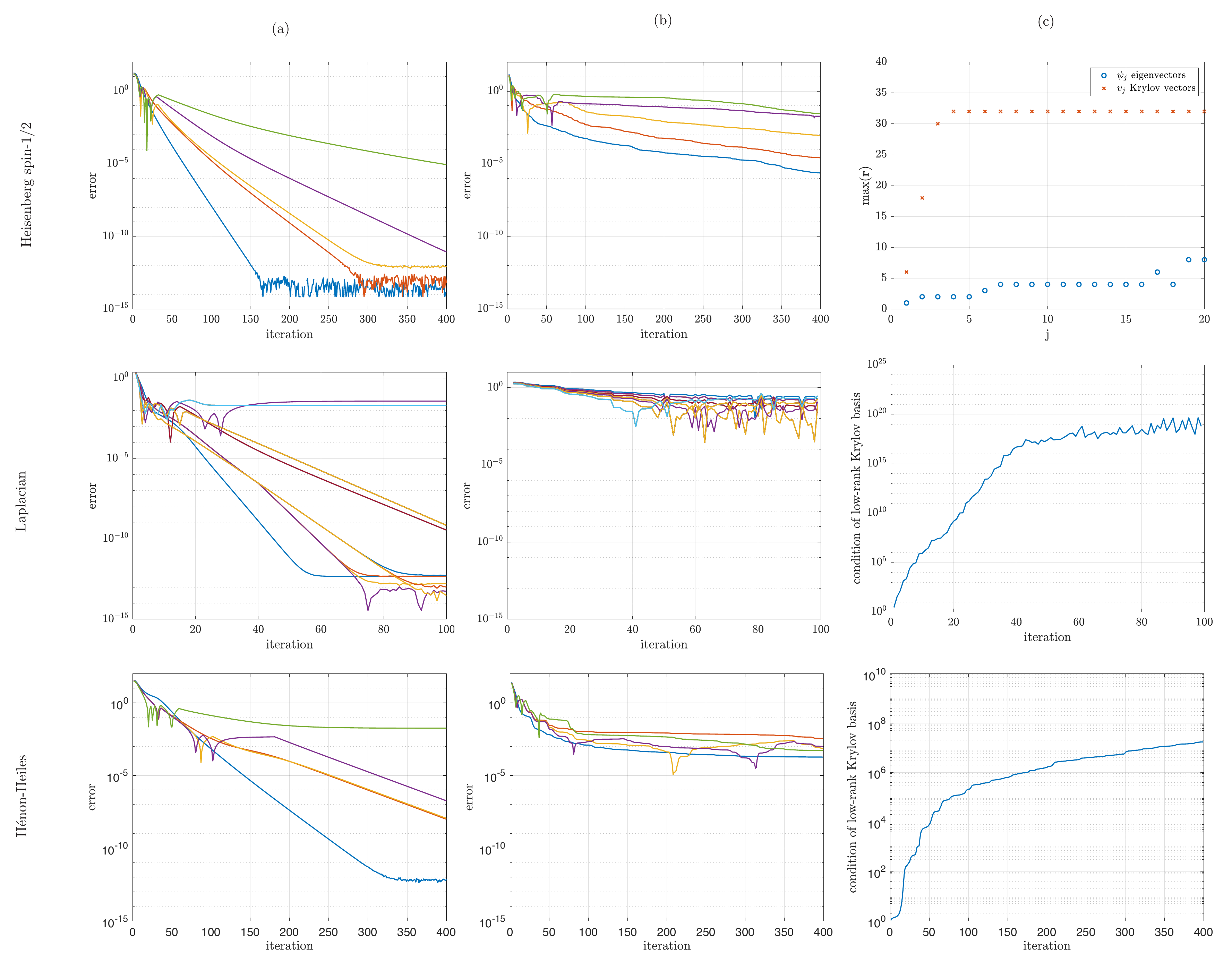}
\caption{Comparison of convergence of the rank-truncated subspace iteration and Lanczos method for each test problem \eqref{Heis}-\eqref{henon_heiles}. 
Top row: Heisenberg Hamiltonian \eqref{Heis} with $L=10$. The errors of the first five Ritz values are shown in column (a) for subspace iteration and column (b) for Lanczos method, both using a fixed truncation rank of 6. For the subspace iteration, subspace dimension $m=5$ and polynomial filter of degree $k=2$ were used. Shown in the top row of column (c) are the ranks needed to represent the exact eigenvectors $\psi_j$ (blue) and the orthogonal Krylov basis vectors $v_j$ with accuracy $10^{-10}$ in the Frobenius norm (red). Also shown are the condition numbers of the approximate low-rank Lanczos basis at each iteration. 
Middle row: Laplacian with $d=3$, $n=16$, truncation rank $1$. The second row of column (c) shows the condition number of the low-rank approximate Krylov basis at each iteration. 
Bottom row: Hamiltonian with H\'enon Heiles potential with $d=3$, $n=16$ and truncation rank 10. Column (c) shows the condition number of the low-rank approximate Krylov basis at each iteration. 
}
\label{fig:heis_krlv_sub_eig_conv}
\end{figure}

We compared the rank truncated subspace iteration (\cref{alg:truncated_power}) and Lanczos method (\cref{alg:TT-Lanczos}) for computing the first several smallest eigenvalues of each test problem \eqref{Heis}-\eqref{henon_heiles}. For \eqref{Heis} we set $L=10$ resulting in $H$ with size $2^{10} \times 2^{10}$, small enough to compute eigenvalues and eigenvectors using standard methods, which serve as benchmarks for evaluating the low-rank methods. \cref{fig:heis_krlv_sub_eig_conv} shows the error versus iteration for the first $5$ Ritz values computed by the low-rank subspace iteration and the low-rank Lanczos method. While the convergence Lanczos method stagnates around $10^{-3}$, the subspace iteration achieves convergence to machine precision. 
We also ran the Lanczos method without rank truncation to obtain the orthogonal Krylov basis vectors that the rank-truncated Lanczos method aims to approximate. In \cref{fig:heis_krlv_sub_eig_conv}(c), we plot the maximum of the TT-rank vector required to represent the first $20$ Krylov basis vectors with accuracy $10^{-10}$. After just two iterations, this rank exceeds $30$ indicating that these Krylov basis vectors can not be efficiently represented with small TT rank. Furthermore, Krylov methods are known to be sensitive to large errors in the early iterations, indicating that the stagnation in convergence arises from the inability to accurately represent the orthogonal Krylov basis vectors as low-rank TTs. 
In contrast, the subspace iteration does not encounter convergence issues due to rank truncation. Unlike the Krylov basis vectors, the ground state and first excited states can be effectively represented as low-rank TTs, as shown in the right panel of \cref{fig:heis_krlv_sub_eig_conv}. Moreover the subspace iteration is robust to large truncation errors in the initial iterations. 

For the Laplace \eqref{eq:laplacian-operator} and Hamiltonian with H\'enon-Heiles potential \eqref{henon_heiles} we set $d=3$ and $n = 16$ resulting in problems of size $16^3 \times 16^3$, small enough to compute eigenvalues and eigenvectors using standard methods, which serve as benchmarks for evaluating the low-rank methods. \cref{fig:heis_krlv_sub_eig_conv} shows the error versus iteration for the first several Ritz values computed by the low-rank subspace iteration and the low-rank Lanczos method. Once again the convergence Lanczos method stagnates around $10^{-3}$, while the subspace iteration achieves convergence to machine precision for eigenvectors that admit low-rank representations. In column (c), we plot the condition number of the approximate low-rank Krylov basis at each Lanczos iteration, and observe that they increase significantly with respect to the iteration number. Such observation partially explains why the convergence of the rank truncated Lanczos iteration stagnates.

\subsection{Comparing subspace iteration and DMRG} 
\begin{figure}[tbhp]
\centering
\includegraphics[scale=0.25]{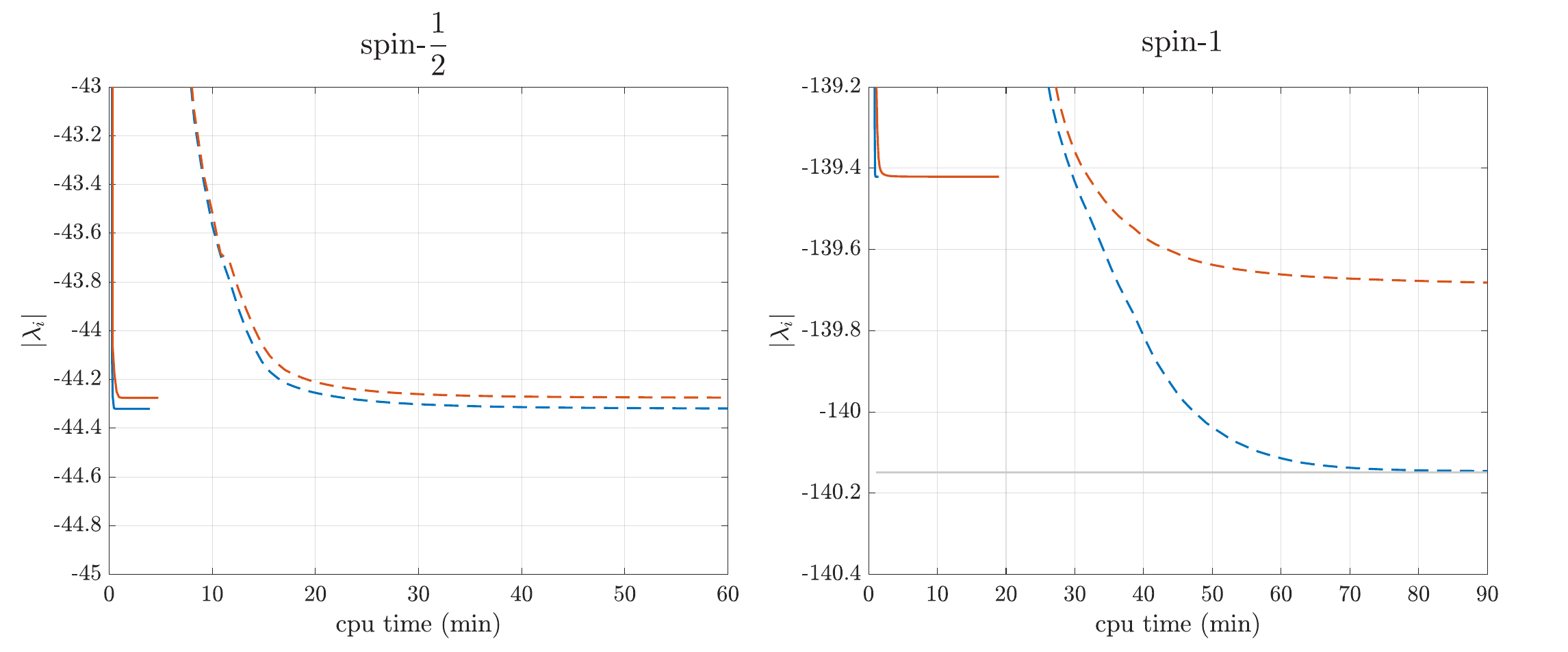}
\caption{Ground state (blue) and first excited state (red) energies of spin-1/2 (left) and spin-1 (right) Heisenberg Hamiltonians \eqref{Heis} with $L=100$, $J=1$, $h=0$ and periodic boundary conditions. We computed the energies using rank-truncated subspace iteration (dashed line) and DMRG method (solid line) and plot the results versus CPU-time. On the right we show the reference ground state energy -140.14840390392 \cite{White2005}. 
}
\label{fig:dmrg_sub_comp}
\end{figure}
We now compare rank-truncated subspace iteration with DMRG on spin-1/2 and spin-1 Heisenberg Hamiltonian \eqref{Heis} with $L=100$, no external magnetic field ($h=0$), $J=-1$ and periodic boundary conditions. 
We tested the DMRG implementation block2 \cite{zhai2023block2} and the Matlab TT-Toolbox implementation \cite{dolgov2014}. Both converged to the same energies and the results reported below were obtained with block2. 
We computed the ground state and first excited state energies (two algebraically smallest eigenvalues) using subspace iteration and DMRG method and plot the energies versus CPU-time in \cref{fig:dmrg_sub_comp}. In both cases we used a maximum rank of 100 for all vectors and both methods are initialized with the same random initial vectors. 
For the spin-1/2 system, the DMRG method and subspace iteration converge to the same eigenvalues with DMRG converging significantly faster. 
For the spin-1 system, DMRG also converges faster but gets stuck in a local minimum, a behavior that has been previously reported \cite{White2005}. In contrast, subspace iteration achieves lower energies, with the ground state energy converging to the reference value. 
Note that the subspace correction for DMRG proposed in \cite{White2005} converges to the ground state for the spin-1 system considered here. The correction enriches the local MPS basis of the current iterate $v$ with components of $Av$, similar to power iteration which uses $\mathfrak{T}_{\bm r}(Av)$ as the updated vector.

\subsection{Demonstration of acceptable truncation error} 

To demonstrate the upper bounds for truncation error in \eqref{eq:err_bnd1} and \cref{thm:pwr_convergence}, \cref{thm:sub_conv} which ensure truncated subspace iteration converges, we considered the Laplacian \eqref{eq:laplacian-operator} with $d=2$ and $n=32$. We used a Chebyshev polynomial filter of degree $k=8$, resulting in an eigenvalue ratio $|p_k(\lambda_2)/p_k(\lambda_1)|\approx 0.6807$. At each iteration we computed the upper bound in \eqref{eq:err_bnd1} for the acceptable truncation error and adaptively selected the rank using TT-SVD so that \eqref{eq:err_bnd1} is satisfied. \cref{fig:laplace_2d_conv}(a) shows the computed upper bound for the truncation error and the actual truncation error achieved using the rank selected by the TT-SVD truncation algorithm. \cref{fig:laplace_2d_conv}(b) shows the rank selected by the TT-SVD truncation at each iteration. In early iterations, achieving the small truncation errors demands a larger rank. This is because applying the operator $p_k(A)$ to a random initial vector generates a vector that require higher ranks for an accurate representation. As the algorithm converges, small truncation errors are again required; however, this can be achieved with a smaller rank since the dominant eigenvector $\psi_1$ has rank $1$. Finally, \cref{fig:laplace_2d_conv}(c) shows the angle $\theta$ of the truncated power iteration vector with the dominant eigenvector $\psi_1$, demonstrating that the method is converging. 

\begin{figure}[tbhp]
\centering
\includegraphics[scale=0.20]{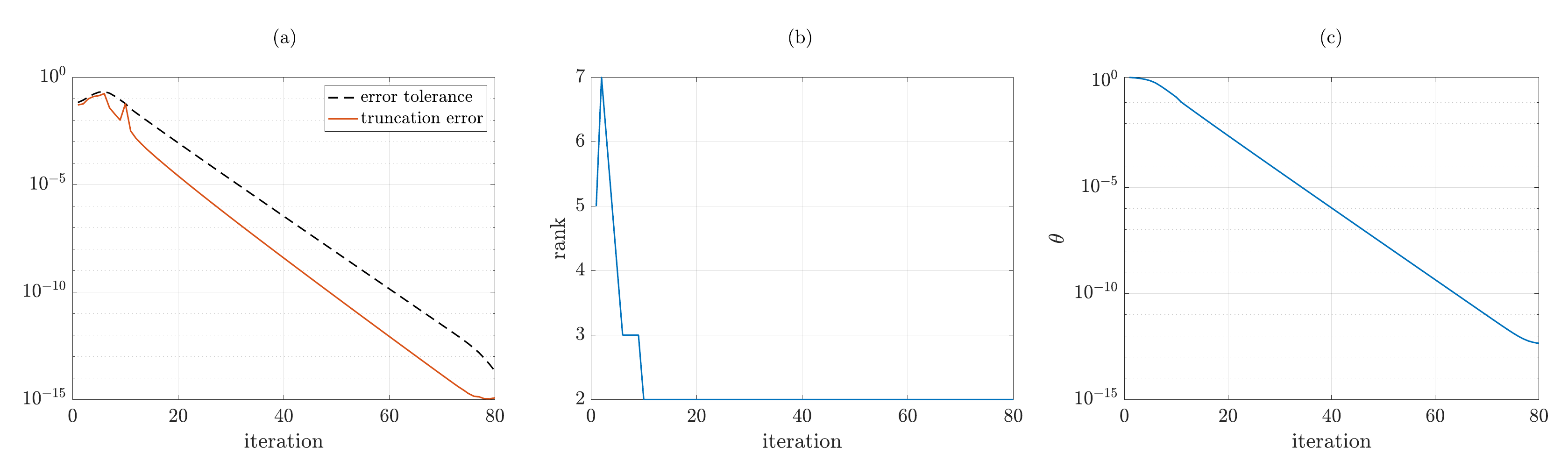} 
\caption{
Numerical demonstration of the sufficient condition for convergence of truncated power iteration provided in \eqref{eq:err_bnd1} for Laplacian \eqref{eq:laplacian-operator} with polynomial acceleration. We set dimension $d=2$, mode sizes $n=32$, and polynomial filter degree $k=8$ with appropriately chosen linear map \eqref{eq:linear_trans}, which yields an eigenvalue ratio $|p_k(\lambda_2)/p_k(\lambda_1)|\approx 0.6807$. (a) Upper bound \eqref{eq:err_bnd1} used as error tolerance in TT-SVD truncation and the resulting truncation error. (b) Truncation rank adaptively selected by TT-SVD truncation to achieve truncation error smaller than the tolerance. (c) Angle between current iteration and dominant eigenvector. 
}
\label{fig:laplace_2d_conv}
\end{figure}

In \cref{fig:laplace_2d_conv_iters}, we plot the projections \eqref{eq:R2_proj} onto $\mathbb{R}^2$ of truncated power iterates 3, 8, 14 before and after applying low-rank truncation. For each of these iterations, we show the projection of the iterate $v^{(i-1)}$ (black arrow), the projection of $p_k(A)v^{(i-1)}$ which is the next iterate before truncation (blue arrow), and the projection of $v^{(i)} = \T_{\r}(p_k(A)v^{(i-1)})$ (red arrow). We also computed the upper bounds from \eqref{eq:err_bnd1} and \cref{thm:pwr_convergence}. Circles with radii corresponding to these upper bounds are shown, with the green circle corresponding to \eqref{eq:err_bnd1} and the red circle corresponding to \cref{thm:pwr_convergence}. The truncation error upper bound given in \cref{thm:pwr_convergence} is more restrictive than the bound in \eqref{eq:err_bnd1}, hence the red circle is smaller than the green circle. As the iterates converge to $\psi_1$, these two bounds become closer to each other. We selected the rank of the truncated vector $v^{(i)}$ adaptively using the TT-SVD truncation algorithm~\cite{oseledets2011tensor} to ensure that the truncation error satisfies \eqref{eq:err_bnd1}, i.e, the projected $v^{(i)}$ (red arrow) remains within the green circle. 
In this case, we observe that truncation adversely affects convergence during early iterations when $\measuredangle(v^{(i)},\psi_1)$ is close to $\pi/2$, such as iteration 3. However, when $\measuredangle(v^{(i)},\psi_1)$ is close to $\pi/4$, truncation can have a beneficial effect. Notably, in iteration 8, truncation reduces the angle between the iterate and $\psi_1$, thereby accelerating convergence during that iteration. 
In early iterations, when $\measuredangle(v^{(i)},\psi_1)$ is close to $\pi/2$, only a small truncation error can be tolerated because the application of the operator $p_k(A)$ to $v^{(i)}$ reduces the angle between $v^{(i)}$ and $\psi_1$ by only a small amount. As $\measuredangle(v^{(i)},\psi_1)$ approaches $\pi/4$, the operator $p_k(A)$ becomes more effective at decreasing the angle, allowing for larger truncation errors. Near convergence, when $\measuredangle(v^{(i)},\psi_1)$ is close to zero, small truncation errors are once again needed to maintain accuracy. This behavior aligns with our analysis in \cref{thm:pwr_convergence} and in particular is consistent with the curve shown in the right panel of \cref{fig:power_conv}.

\begin{figure}[tbhp]
\centering
\includegraphics[scale=0.20]{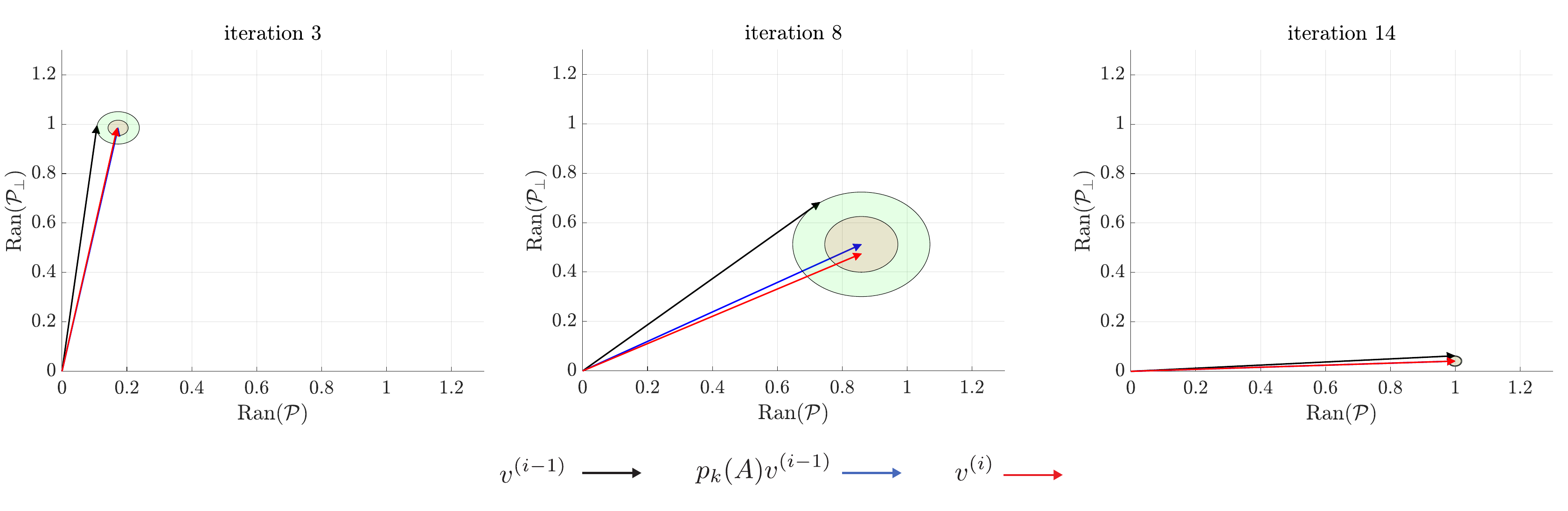} 
\caption{
Numerical demonstration of the truncation error upper bounds in \eqref{eq:err_bnd1} and \cref{thm:pwr_convergence} during the computation of the dominant eigenvector of the Laplacian \eqref{eq:laplacian-operator} with dimension $d=2$ and $n=32$ points per dimension. A Chebyshev polynomial filter of degree $8$ was used, yielding an eigenvalue ratio $|p_k(\lambda_2)/p_k(\lambda_1)|\approx 0.6807$. The figure shows the projections of iterates onto $\mathbb{R}^2$, as defined in \eqref{eq:R2_proj}, at various iterations both before and after truncation. 
}
\label{fig:laplace_2d_conv_iters}
\end{figure}

\subsection{Methods for improving efficiency} 
\label{sec:improving_eff}

\begin{table}[tbhp]
\footnotesize
\caption{Timing and convergence of rank-truncated subspace iteration for Heisenberg spin-1/2 \eqref{Heis} with $L=32$. We set maximum rank equal to $2$ and report average CPU-time per iteration and the number of iterations until the residual norm of the second eigenvector is less than $10^{-10}$. If the second eigenvector has not converged in $5000$ iterations we report its error in the residual norm.} 
\label{table:timings_heis}
\begin{center}
\begin{tabular}{ c | c | c | c | c | c}
k (poly. deg) & m (sub. dim) & avg. CPU-time per iter & \# iters & total CPU-time & residual norm \\ 
\hline 
none & 2 & 0.02s & 5000 & 100s & $8.87 \times 10^{-3}$ \\
none & 4 & 0.05s & 5000 & 250s & $8.85 \times 10^{-5}$ \\
none & 8 & 0.15s & 5000 & 750s & $1.06 \times 10^{-10}$ \\
2 & 2 & 0.03s & 5000 & 150s & $5.00 \times 10^{-3}$ \\
2 & 4 & 0.08s & 5000 & 400s & $3.90 \times 10^{-5}$ \\ 
2 & 8 & 0.21s & 3293 & 692s & $<10^{-10}$ \\
4 & 2 & 0.05s & 5000 & 250s & $9.22 \times 10^{-4}$ \\
4 & 4 & 0.12s & 5000 & 600s & $5.32 \times 10^{-9}$ \\ 
4 & 8 & 0.28s & 1331 & 373s & $<10^{-10}$ \\
8 & 2 & 0.09s & 5000 & 450s & $2.53 \times 10^{-5}$ \\
8 & 4 & 0.20s & 3106 & 621s & $<10^{-10}$ \\ 
8 & 8 & 0.44s & 681 & 300s & $<10^{-10}$ \\
\end{tabular}
\end{center}
\end{table}

To compare the choice of subspace size and polynomial filter degree in rank-truncated subspace iteration, we considered a Heisenberg spin-1/2 chain of length $L=32$ and computed the first excited state using subspaces of dimensions $m=2,4,8$ and polynomial filters of degrees $k=2,4,8$. Increasing either the subspace dimension or the polynomial filter degree results in more costly iterations. However, similar to classical subspace iteration, increasing these parameters accelerates convergence. For all experiments we truncate with a maximum rank of $2$ during each iteration. \cref{table:timings_heis} presents the average CPU-time per iteration, the number of iterations, and the residual of the first excited state. 

Finally, we compare implementations of the rank-truncation operator $\mathfrak{T}_{\bm r}$ for improving the computational efficiency of rank-truncated subspace iteration. First, we compare randomized algorithms with TT-SVD truncation to compute low-rank approximations of linear combinations in the Rayleigh-Ritz step as described in \cref{sec:randtruncation}. Second, we compare tangent space projections with TT-SVD for approximating matrix-vector products and linear combinations. 

\begin{figure}[t]
\centering
\includegraphics[width=0.78\textwidth]{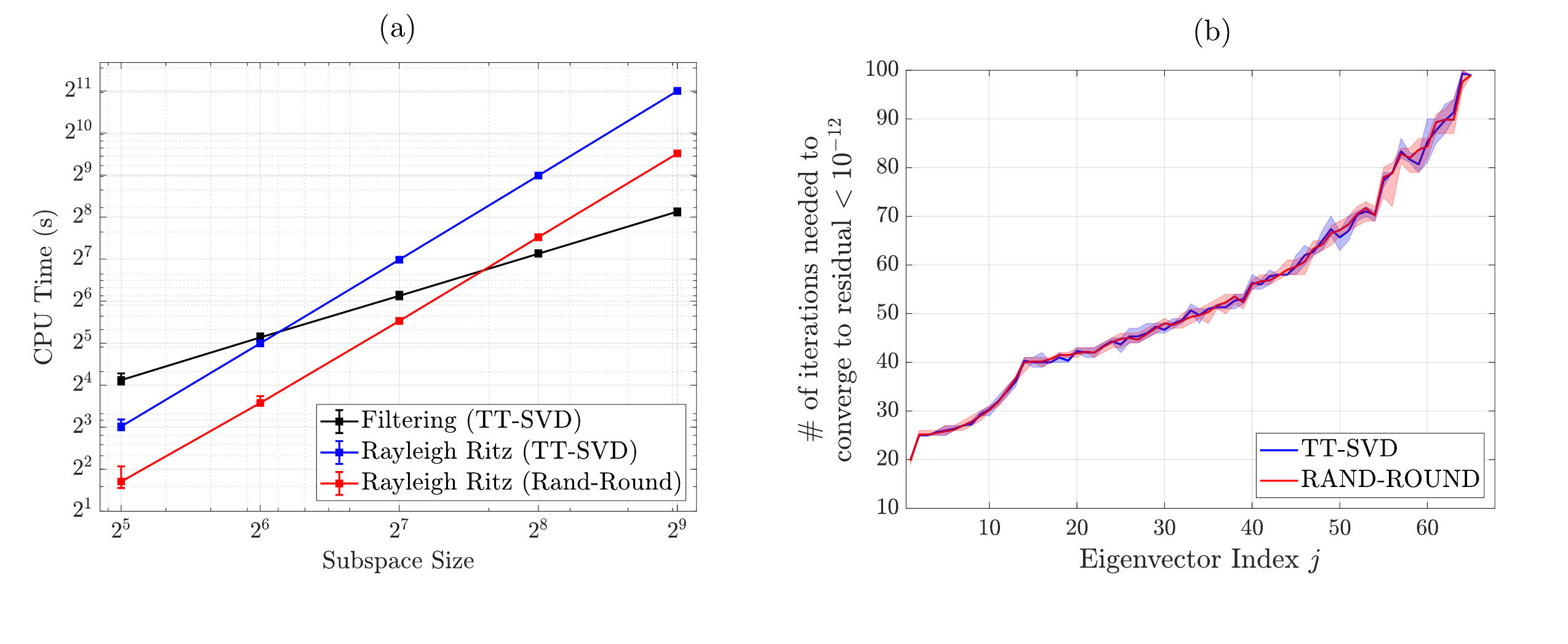}
\caption{Comparison between TT-SVD and randomized truncation methods \cite{Daas2021} for approximating linear combinations in the Rayleigh-Ritz step of subspace iteration. We considered the spin-1/2 Heisenberg chain of length $L=24$, iterating at fixed TT-rank $r = 16$. For both methods, TT-SVD is used for truncations during polynomial filtering. 
(a) Comparison of cpu time. Using of the randomized truncation methods consistently reduces the per-iteration run time of the Rayleigh-Ritz step by more than 2.5x. This speed up is particular important for larger subspaces, for which the Rayleigh-Ritz step requires increasingly more time than polynomial filtering. 
(b) Comparison of convergence. number of iterations required for each eigenvector estimate to converge in residual $\norm{Av_j / \langle v_j,Av_j \rangle - v_j}_F < 10^{-12}$ for a subspace of size $m = 96$. We plot the min, mean, and max over 6 different initial subspaces. Despite the potential for additional error due to randomization, using randomized methods for truncation during the Rayleigh Ritz step prolongs convergence by no more than 3 iterations on average. 
}
\label{fig:ttsvd_vs_rand}
\end{figure}

To demonstrate the computational advantage of using randomized rounding algorithms for the low-rank Rayleigh-Ritz step, we considered the Heisenberg spin-1/2 \eqref{Heis} with $L = 24$, using a polynomial filter of degree $k = 4$ and a maximum TT-rank of $16$. 
Figure \ref{fig:ttsvd_vs_rand}(a) demonstrates the per-iteration speed-up when randomization is used during the Rayleigh-Ritz step.
Truncations during polynomial filtering dominate the run time for smaller subspaces $m$ due to the need to round matrix-vector products. 
The Rayleigh Ritz procedure becomes the primary cost for subspace sizes beyond $m = 128$, and in this parameter regime, randomization improves the time efficiently of the RR step by more than 2.5 times. Although randomization may introduce larger truncation errors than TT-SVD, we find that randomized rounding during the Rayleigh Ritz step has minimal effect on converge rate. 
Figure \ref{fig:ttsvd_vs_rand}(b) shows the number of iterations required for each eigenvector estimate to converge in residual $\norm{Av_j / \langle v_j,Av_j \rangle - v_j}_F < 10^{-12}$ for a subspace of size $m = 96$. We plot the min, mean, and max over 6 different initial subspaces. 
We see considerable overlap between the curves for the Rayleigh-Ritz step with and without randomization (blue vs. red, respectively): on average, the randomized method requires no more than 3 iterations to converge to the same tolerance of the deterministic TT-SVD.
With randomization decreasing the per-iteration cost by more than 2 times, randomized rounding results in better overall time efficiency despite the need for a few extra iterations.

\begin{figure}[t]
\centering
\includegraphics[scale=0.23]{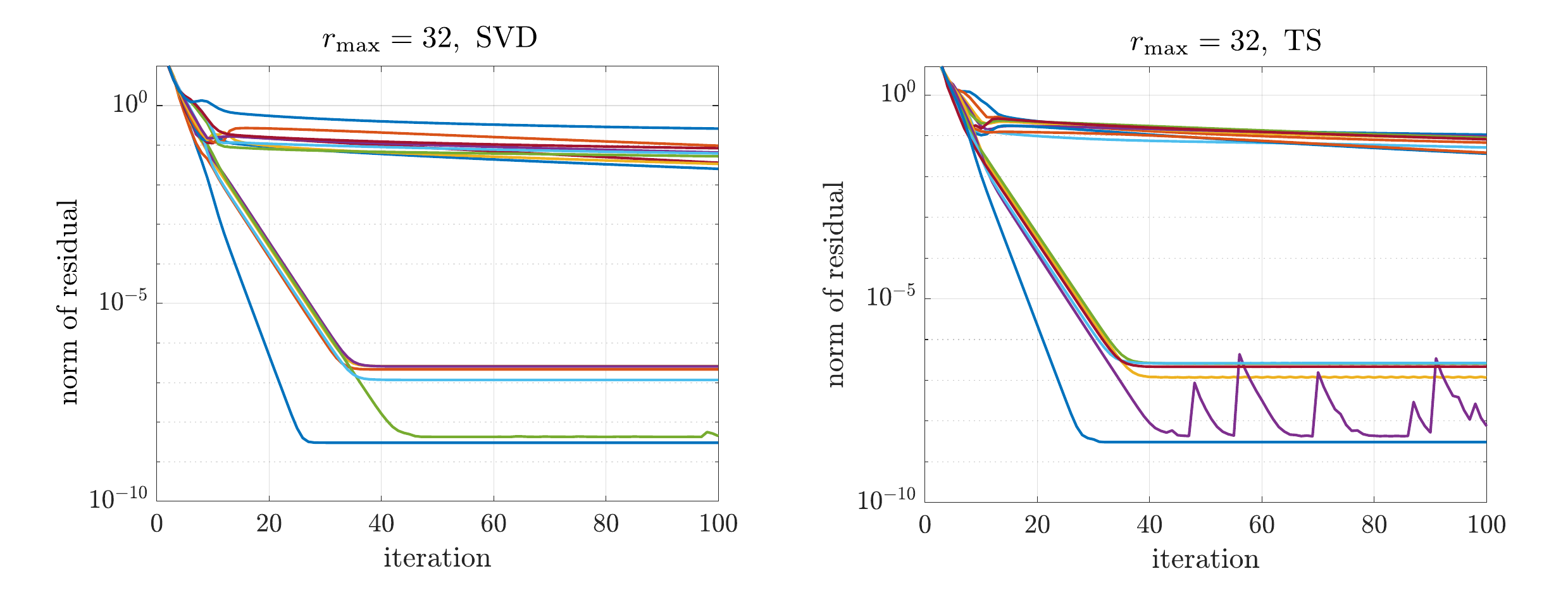}
\caption{
Convergence of rank truncated subspace iteration for computing the first several excited states of the H\'enon-Heiles Hamiltonian \eqref{henon_heiles} with dimension $d=5$ discretized with $n=28$ spectral collocation points per dimension. We used a subspace of dimension $m=15$, Chebyshev polynomial filter of degree $k=6$, and compare convergence for maximum rank $r_{\textrm{max}} = 32$. We show results obtained with two low-rank truncation strategies for approximating matrix-vector products and linear combinations. Truncation performed with rank increase followed by SVD (left) and truncation by approximating matrix-vector products and linear combinations in appropriate tangent spaces of the low-rank manifold (right). The $100$ iterations using TT-SVD truncation took 1896 seconds while $100$ iterations with tangent space projections took 656 seconds, more than 2.5x speed-up with no significant effect on convergence. }
\label{fig:hh_d5}
\end{figure}

To demonstrate the computational advantage of approximating matrix-vector products and linear combinations using orthogonal projection onto TT tangent spaces as described in \cref{sec:tangent_proj} we considered the Hamiltonian \eqref{henon_heiles} as a TT matrix \eqref{eq:tt_matrix} with accuracy $10^{-10}$ yielding TT matrix rank $\bm r_{A} = \begin{bmatrix} 1 & 4 & 4 & 4 & 4 & 1 \end{bmatrix}$. The subspace dimension was set to $m=15$, and a Chebyshev polynomial filter of $k=6$ was used. The method was run with maximum truncation rank $r_{\textrm{max}} = 32$ and different rank truncation strategies. First we used TT-SVD truncation to approximate matrix-vector products and linear combinations of TTs. Such approach is computationally expensive since, e.g., matrix-vector products between $H$ and $v$ result in TTs with ranks equal to the product of the ranks of $H$ and $v$ which are then truncated costing $\mathcal{O}(dnr_v^3r_A^3)$, where $r_v$ and $r_A$ are the entries of the corresponding TT-rank vectors (here assumed to be constant). Then we ran the subspace iteration and approximated matrix-vector products and linear combinations in appropriate tangent spaces (TS) of low-rank TT manifolds. This approach avoids rank increase and is therefore more efficient, e.g., with approximate matrix-vector products between $A$ and $v$ costing $\mathcal{O}(dnr_v^3)$ where $Av$ is approximated in the tangent space of the low-rank manifold at $v$. 
\cref{fig:hh_d5} shows the norm of the residual for each Ritz vector at each iteration. As the maximum truncation rank increases, the final accuracy of the Ritz vectors improves. Notably, different truncation ranks and truncation strategies (SVD or TS) do not significantly affect the convergence rate of the method, even when they influence the final accuracy. These observations highlight the robustness of the subspace iteration convergence rate to rank truncation errors.
The $100$ iterations using TT-SVD truncation took 1896 seconds while $100$ iterations with tangent space projections took 656 seconds, more than 2.5x speed-up.

\section{Conclusions}
\label{sec:conclusions}

We examined low-rank inexact variants of the Lanczos method and a polynomial filtered subspace iteration for tensor eigenvalue problems. These methods typically avoid the issue of getting stuck at local minimizers, which can occasionally occur in DMRG-based algorithms. Furthermore, these methods can be easily used to compute more than one eigenpair of a general matrix, whereas the original DMRG algorithm is designed to compute the lowest eigenvalue (i.e., the ground state) of a Hermitian matrix. 
Our results demonstrated that orthogonal Krylov basis vectors often cannot be accurately represented by low-rank tensor trains, even when the eigenvectors of the associated matrix exhibit low-rank structure. We further showed that the low-rank approximation errors from representing orthogonal Krylov basis vectors resulted in stagnation of convergence in the Lanczos method. 
In contrast, the subspace iteration method employs Ritz vectors as the subspace basis, which directly approximate the low-rank eigenvectors of interest. We provided a convergence analysis quantifying the acceptable truncation error that ensures progress toward convergence at each iteration. Numerical experiments further confirmed that the subspace iteration method exhibits significantly greater robustness to truncation errors compared to the Lanczos method. We provided numerical demonstrations of the convergence analysis and showed that polynomial filtered subspace iteration can be used to compute several dominant eigenpairs with machine precision for problems with up to $100$ dimensions, provided the eigenvectors admit low-rank representations. We also demonstrated that rank-truncated subspace iteration can converge for problems where the density matrix renormalization group method (DMRG). Although the analysis presented here focuses on polynomial filtered subspace iteration, it can be easily extended to other type of subspace iteration such as the one based on the FEAST algorithm~\cite{DMRGFEAST}. 

\section*{Acknowledgments}
This material is based upon work supported by the U.S. Department of Energy, Office of Science, Office of Advanced Scientific Computing Research, Scientific Discovery through Advanced Computing (SciDAC) program through the FASTMath Institute and Partnership with Basic Energy Science under U.S. Department of Energy Contract No. DE-AC02-05CH11231.  This research used resources of the National Energy Research Scientific Computing Center, a DOE Office of Science User Facility supported by the Office of Science of the U.S. Department of Energy under Contract No. DE-AC02-05CH11231 using NERSC award ASCR-ERCAPm1027.

\section*{Code availability}
Source code to reproduce numerical experiments in this paper are available at \url{https://github.com/adektor/TT_subspace_eig}.

\bibliographystyle{siam}
\bibliography{refs}

\begin{appendices}

\section{Lanczos method}
\label{sec:krylov}

Krylov subspace methods are a widely used class of iterative algorithms for computing a few extremal eigenvalues of a matrix $A$. Starting from an initial vector $v_0 \in \mathbb{C}^{n_1 \cdots n_d}$, these algorithms construct an orthonormal basis of the Krylov subspace 
\begin{equation} \label{eq:krylov}
\mathcal{K}(A,v_0;m) = \text{span}\{v_0, Av_0,...,A^{m-1}v_0\}.
\end{equation} 
When the matrix $A$ is Hermitian, an orthonormal basis for the Krylov subspace \eqref{eq:krylov} can be generated by a three-term recurrence known as the Lanczos method. If we use $V_m$ to denote the matrix that contains the orthonormal basis (here we are considering the basis tensors as column vectors in $\mathbb{C}^{n_1\cdots n_d}$), the Lanczos method can be succinctly described by the recurrence relation 
\begin{equation}
A V_m = V_m T_m + fe_m^{\top},
\label{eq:lanfact}
\end{equation}
where $T_m = V_m^{\top} A V_m$ is a tridiagonal matrix representing the projection of $A$ onto the Krylov subspace, and the residual vector $f = (I-V_mV_m^{\top})A v_m$ satisfies $V_m^{\top} f = 0$. The tridiagonal structure of $T_m$ indicates that each column of $V_m$ can be generated via a three-term recurrence 
\begin{equation} \label{eq:3term}
\beta_{j+1} v_{j+1} = A v_j - \alpha_j v_j - \beta_{j}v_{j-1}, 
\end{equation}
where $\alpha_j = \langle v_j, A v_j \rangle$ and $\beta_j = \langle v_{j-1}, A v_j\rangle$. It is well known that using such a three-term recurrence to implement the Lanczos algorithm can result in significant loss of orthogonality among the basis vectors $v_j$ \cite{Simon1984}. To address this, conventional implementations often incorporate full or partial/selective reorthogonalization of the basis vectors. This reorthogonalization mitigates the effects of numerical rounding errors and the resulting loss of orthogonality in the Krylov basis, thereby ensuring the numerical stability and accuracy of the method.

\subsection{Lanczos with rank truncation}
\label{sec:lanczos_lr}

For moderate to large $d$, storing dense representations of the matrix $A$ and Krylov basis vectors becomes impractical. To address this, we consider a low-rank TT variant of the Lanczos method that incorporates truncation into the three-term recurrence \eqref{eq:3term}, ensuring the generated basis vectors $v_j$ and intermediate vectors maintain low-rank TT structure. Truncations are applied after matrix-vector multiplications and additions of TTs. The main steps of this approach are summarized in \cref{alg:TT-Lanczos}. We emphasize that rank truncation following the application of $A$ (lines 1 and 7) prevents the construction of an exact Krylov subspace. Similarly, truncations in the two- and three-term recurrences (lines 3 and 9) further deviate the approximate basis from a true Krylov basis and make it nearly impossible to maintain orthonormality among the basis vectors. Managing both orthogonality and low-rank structure presents significant challenges for the Lanczos method. While rank truncation is essential for preserving computational efficiency and a small memory footprint, it inherently leads to a loss of orthogonality, rendering conventional reorthogonalization techniques unsuitable for low-rank TT variants. 

To address the loss of orthogonality in the low-rank Lanczos algorithm, one can apply the approach from \cite{MPS_Lanczos}, where an orthonormal basis of the approximate Krylov subspace is implicitly generated via Cholesky factorization of the Gram matrix $G = V_m^{\top} V_m$, i.e. $G = R^{\top} R$ where $R \in \mathbb{C}^{m \times m}$. By projecting $A$ onto the subspace spanned by columns of $Q_m = V_mR^{-1}$, one obtains the reduced matrix 
\begin{equation} \label{eq:lanczos_reduced_mat}
\tilde{T}_m = R^{-\top}\left( V_m A V_m \right) R^{-1}, 
\end{equation} 
whose eigenvalues (Ritz values) approximate the eigenvalues of $A$, and whose eigenvectors serve as coefficients for approximate eigenvectors (Ritz vectors) of $A$ relative to the basis $V_m$. While this projection method reduces spurious eigenvalues and enhances stability, it does not address the degradation of subspace quality introduced by rank truncation. 

The core issue of applying Krylov subspace methods with low-rank compression lies in the fact that the orthonormal basis vectors of a Krylov subspace are typically not low-rank. Consequently, rank truncation applied at each iteration to manage tensor ranks introduces significant errors, causing the subspace generated by the low-rank Lanczos process to deviate substantially from the true Krylov subspace. As a result, as demonstrated in \cref{sec:numerics}, approximations of eigenvalues and eigenvectors obtained from projecting onto this subspace fail to exhibit the same convergence properties as those achieved by the standard Lanczos method. 

\begin{algorithm}
\caption{Lanczos iteration with rank truncation}
\label{alg:TT-Lanczos}
\begin{algorithmic}[1]
\Require
\Statex \hspace{1em} $A \in \mathbb{C}^{n_1\cdots n_d \times n_1\cdots n_d}$ in TT-matrix format
\Statex \hspace{1em} $v_0 \in \mathbb{C}^{n_1\cdots n_d}$ TT with $\|v_0\|=1$
\Statex \hspace{1em} $m \in \mathbb{N}$ Number of Lanczos steps
\Statex \hspace{1em} $\bm r \in \mathbb{N}^{d+1}$ TT truncation rank
\Ensure
\Statex \hspace{1em} Approximate Krylov basis vectors $(v_0, v_1, \ldots, v_{m-1})$ in TT format
\vspace{0.5em}
\State $v_1 \gets \mathfrak{T}_{\bm r}(A v_0)$ \Comment{Apply $A$ and truncate}
\State $\alpha_0 \gets \langle v_1, v_0 \rangle$ \Comment{Compute first diagonal element}
\State $v_1 \gets \mathfrak{T}_{\bm r}(v_1 - \alpha_0 v_0)$ \Comment{Orthogonalize and truncate}
\For{$j = 1$ to $m-1$}
    \State $\beta_j \gets \|v_j\|$ \Comment{Compute subdiagonal element}
    \State $v_j \gets v_j / \beta_j$ \Comment{Normalize}
    \State $v_{j+1} \gets \mathfrak{T}_{\bm r}(A v_j)$ \Comment{Apply $A$ and truncate}
    \State $\alpha_j \gets \langle v_{j+1}, v_j \rangle$ \Comment{Compute diagonal element}
    \State $v_{j+1} \gets \mathfrak{T}_{\bm r}(v_{j+1} - \alpha_j v_j - \beta_j v_{j-1})$ \Comment{Orthogonalize and truncate}
\EndFor
\end{algorithmic}
\end{algorithm}


\subsection{Convergence of Lanczos with rank truncation}
In the low-rank Lanczos method (\cref{alg:TT-Lanczos}), truncation is applied after matrix-vector products (steps 1 and 7) and when forming linear combinations of low-rank TTs (steps 3 and 9). The use of these truncations results in the generation of a subspace spanned by the basis
\begin{equation}
V_m = \hat{V}_m + E_m,
\end{equation}
where $\hat{V}_m$ is the orthonormal basis of the Krylov subspace \eqref{eq:krylov} produced from a $m$-step truncation free Lanczos procedure and $E_m$ is a matrix that accounts for all of the truncation errors accumulated in the inexact $m$-step Lanczos iteration. We demonstrate in \cref{sec:numerics} that orthonormal Krylov basis vectors are not low-rank, even when the eigenvectors of the corresponding matrix are low-rank. Consequently, the error matrix $E_m$ is relatively large. 

Projecting $A$ into an orthonormal basis of $V_m$ yields a $m \times m$ projected matrix that can be 
viewed as a perturbation of the tridiagonal matrix $T_m$ in \eqref{eq:lanfact} generated by truncation-free Lanczos iteraion. However, since the perturbation can be large, the eigenvalues of such a projected matrix (the Ritz values relative to the inexact Krylov subspace) can be significantly different from the eigenvalues of $T_m$ (the Ritz values of the Krylov subspace). Moreover, a Ritz vector obtained from the approximate Lanczos procedure cannot be expressed as $f_{m-1}(A)v_0$ for some polynomial $f_{m-1}(\lambda)$ of degree $m-1$. Consequently, the accuracy of the approximation cannot be readily analyzed using the polynomial approximation theory typically employed to study the convergence of the Lanczos algorithm~\cite{parlett}. It is, however, well-known that large errors during the early iterations of the Lanczos method, such as those introduced by rank truncation, can lead to stagnation in convergence~\cite{PAIGE1980235}. This stagnation is observed in our numerical experiments presented in \cref{sec:numerics}. 

Note that the tridiagonal matrix produced by \cref{alg:TT-Lanczos}, with entries $\alpha_j$ and $\beta_j$, is not used to approximate the eigenvalues of $A$. Instead, we project $A$ onto the orthonormal basis ${V}_m$ and perform a Rayleigh-Ritz approximation within this subspace by computing the eigendecomposition of the reduced matrix \eqref{eq:lanczos_reduced_mat}. If a Ritz vector is computed explicitly, additional truncation error is introduced because each Ritz vector is a linear combination of the columns of $V_m$, and such a combination requires truncation to maintain low-rank structure. This error can be avoided by implicitly storing the Ritz vectors. Specifically, this involves storing the basis $V_m$ and the eigenvectors of the projected matrix, which represent the coefficients of the Ritz vectors relative to $V_m$. Quantities that depend on the eigenvectors, such as observables, can then be approximated directly using $V_m$ and the eigenvectors of the projected matrix, eliminating truncation errors introduced during the Rayleigh-Ritz procedure.

\section{Orthogonal projection onto TT tangent space} \label{sec:tangent_proj}
The collection of all TTs with rank-$\r$ 
\begin{equation}
\label{TT-mfld}
\mathcal{M}_{\bm r} = \{v \in \mathbb{C}^{n_1\times \cdots \times n_d}  \mid  \text{TT-rank}(v) = \bm r\},
\end{equation}
forms a smooth submanifold of $\mathbb{C}^{n_1 \times \cdots \times n_d}$ \cite{TT-mfld}. Associated with any $v \in \M_{\r}$ is a tangent space denoted by $T_{v}\M_{\r}$, which is a vector subspace of $\mathbb{C}^{n_1 \times \cdots \times n_d}$. Such tangent spaces, typically used for dynamical low-rank approximation of tensor differential equations \cite{Lubich2015,Haegeman2016}, are a convenient set of vector subspaces in which we can approximate the linear algebra operations needed for inexact subspace projection methods. 
Given the TT representation \eqref{eq:TT_format} of $v$, any element of the tangent space can be written (non-uniquely) as 
\begin{equation}
\label{tangent_element}
    \delta v = \delta C_1 C_2 \cdots C_d + C_1 \delta C_2 C_3 \cdots C_d + \cdots + C_1 \cdots C_{d-1} \delta C_d, 
\end{equation}
where $\delta C_k \in \mathbb{C}^{r_{k-1} \times n_k \times r_k}$ is a first order variations of the TT $C_k$. Approximating a given tensor $z \in \mathbb{C}^{n_1\times \cdots \times n_d}$ in the tangent space $T_v\M_{\r}$ amounts to determining a good set of $\delta C_k$ so that the error $\|z-\delta v\|$ is minimized. When $z$ is a TT, a sum of TTs, or a TT matrix-vector product, the cores $\delta C_k$ that minimizes the Frobenius norm error can be computed efficiently by projecting $z$ orthogonally onto $T_v\M_{\r}$~\cite{Steinlechner2016}. 

\section{Proof of Theorem~\ref{thm:pwr_conv_no_trunc}}
\label{sec:proofpwr}
\begin{proof}
Tangent of the angles $\measuredangle(v,\psi_1)$ and $\measuredangle(Av,\psi_1)$ can be expressed as 
\begin{equation}
\label{eq:angles}
\tan\measuredangle(v,\psi_1) = \frac{\left\|\P_{\perp} v\right\|}{\left\|\P v\right\|} \quad \text{and} \quad \tan\measuredangle(Av,\psi_1) = \frac{\left\|\P_{\perp} Av\right\|}{\left\|\P Av\right\|}. 
\end{equation} 
To derive the bound \eqref{eq:pwr_conv}, we first establish
\begin{equation} \label{eq:projector1_properties}
\left\| \P_{\perp} Av \right\| \leq  |\lambda_2|\left\|\P_{\perp} v\right\|
\quad \text{and} \quad 
\left\| \P Av \right\| = |\lambda_1|\left\|\P v\right\|. 
\end{equation} 
Expanding $v$ in the eigenbasis $v = \displaystyle\sum_{j=1}^{n_1 \cdots n_d} \alpha_j\psi_j$, 
the inequality is obtained from 
\begin{equation}
        \left\|\P_{\perp}A v\right\| = 
         \left\| \sum_{j=2}^{n_1\cdots n_d} \alpha_j \lambda_j \psi_j \right\| \\
        \leq |\lambda_2| \left\| \sum_{j=2}^{n_1\cdots n_d} \alpha_j \psi_j \right\| \\
        = |\lambda_2| \left\| \P_{\perp} v \right\|, 
\end{equation}
where the inequality is due to $|\lambda_2|\geq|\lambda_j|$ for all $j \geq 2$. The equality in \eqref{eq:projector1_properties} is obtained similarly. 
Combining \eqref{eq:angles} and \eqref{eq:projector1_properties} we have 
\begin{equation} \label{eq:power_ratio1}
    \tan\measuredangle(Av,\psi_1) 
    = \frac{\left\|\P_{\perp} Av\right\|}{\left\|\P Av\right\|} 
    \leq \frac{|\lambda_2|\left\|\P_{\perp} v\right\|}{|\lambda_1| \left\|\P v\right\|} \\
    = \left|\frac{\lambda_2} {\lambda_1} \right|\tan\measuredangle(v,\psi_1), 
\end{equation}
completing the proof. 
\end{proof}

\section{Proof of Theorem~\ref{thm:sub_conv}}
\label{sec:proofsub}
\begin{proof}
Using the eigendecomposition $A = \Psi \Lambda \Psi^{\top}$, we write $T_{AV}$ as  
\begin{equation} \label{eq:TAV}
\begin{aligned}
T_{AV} &= \Psi_{>m}^{\top} AV \left(\Psi_{\leq m}^{\top} AV\right)^{-1} \\
&= \Psi_{>m}^{\top} \Psi \Lambda \Psi^{\top} V \left(\Psi_{\leq m}^{\top} \Psi \Lambda \Psi^{\top} V\right)^{-1} \\
&= \Lambda_{>m} \Psi_{>m}^{\top} V \left(\Psi_{\leq m}^{\top} V\right)^{-1} \Lambda_{\leq m}^{-1} \\
&= \Lambda_{>m} T_V \Lambda_{\leq m}^{-1}, 
\end{aligned}
\end{equation}
where $\Lambda_{>m} = \text{diag}(\lambda_{m+1},\ldots,\lambda_{N})$ and $\Lambda_{\leq m} = \text{diag}(\lambda_{1},\ldots,\lambda_{m})$. 
According to \eqref{eq:tan_princ_angle}, principal angle of interest is the largest singular value of $T_{AV}$, which we bound using \eqref{eq:TAV} as 
\begin{equation} \label{eq:p_angle_dec}
\begin{aligned}
\tan \Theta(\R(AV),\R(\Psi_{\leq m})) &= \sigma_1(T_{AV})\\ 
&= \|T_{AV}\|_2 \\
&= \left\| \Lambda_{>m} T_V \Lambda_{\leq m}^{-1} \right\|_2 \\
&\leq \left\| \Lambda_{>m} \right\|_2 \left\| T_V \right\|_2 \left\|\Lambda_{\leq m}^{-1} \right\|_2 \\
&= \left|\frac{\lambda_{m+1}}{\lambda_m}\right| \left\|T_V\right\|_2 \\
&= \left|\frac{\lambda_{m+1}}{\lambda_m}\right| \tan \Theta(\R(V),\R(\Psi_{\leq m})),
\end{aligned}
\end{equation}
completing the proof. 
\end{proof}

\end{appendices}

\end{document}

%% file: shared.tex
\usepackage{hyperref}
\usepackage{amssymb,amsfonts,amsmath,mathrsfs,mathtools}
\usepackage{graphicx,epsfig,subfigure}
\usepackage{bm}
\usepackage{times}
\usepackage{multirow}
\usepackage{color}
\usepackage{float}
\usepackage{algorithm}%
\usepackage{algorithmicx}%
\usepackage{algpseudocode}%
\usepackage{listings,comment}%
\usepackage[title]{appendix}%

\renewcommand{\top}{\text{T}}
\def\T{\mathfrak{T}}
\def\vs{\vspace{0.2cm}}

\def\P{\mathcal{P}}

\def\M{\mathcal{M}}
\def\R{\mathcal{R}}

\def\r{\bm r}

\newcommand{\norm}[1]{\left \lVert #1 \right \rVert }
\newcommand{\Plm}{\Psi_{\le m}}
\newcommand{\Pgm}{\Psi_{> m}}
\newcommand{\Llm}{\Lambda_{\le m}}
\newcommand{\Lgm}{\Lambda_{> m}}

%% file: main.bbl
\begin{thebibliography}{10}

\bibitem{al2023randomized}
{\sc H.~Al~Daas, G.~Ballard, P.~Cazeaux, E.~Hallman, A.~Mi{k{e}}dlar, M.~Pasha,
  T.~W. Reid, and A.~K. Saibaba}, {\em Randomized algorithms for rounding in
  the tensor-train format}, SIAM J. Sci. Comput., 45 (2023), pp.~A74--A95.

\bibitem{DMRGFEAST}
{\sc A.~Baiardi, A.~K. Kelemen, and M.~Reiher}, {\em Excited-state {DMRG} made
  simple with {FEAST}}, J. Chem. Theory Comput., 18 (2022), pp.~415--430.
\newblock PMID: 34914392.

\bibitem{ballani2013projection}
{\sc J.~Ballani and L.~Grasedyck}, {\em A projection method to solve linear
  systems in tensor format}, Numer. Linear Algebra Appl., 20 (2013),
  pp.~27--43.

\bibitem{Daas2021}
{\sc H.~A. Daas, G.~Ballard, P.~Cazeaux, E.~Hallman, A.~Miedlar, M.~Pasha,
  T.~W. Reid, and A.~K. Saibaba}, {\em Randomized algorithms for rounding in
  the tensor-train format}, SIAM J. Sci. Comput., 45 (2021), pp.~74--.

\bibitem{MPS_Lanczos}
{\sc P.~E. Dargel, A.~W\"ollert, A.~Honecker, I.~P. McCulloch, U.~Schollw\"ock,
  and T.~Pruschke}, {\em Lanczos algorithm with matrix product states for
  dynamical correlation functions}, Phys. Rev. B, 85 (2012), p.~205119.

\bibitem{dolgov2014}
{\sc S.~V. Dolgov, B.~N. Khoromskij, I.~V. Oseledets, and D.~V. Savostyanov},
  {\em Computation of extreme eigenvalues in higher dimensions using block
  tensor train format}, Comput. Phys. Commun., 185 (2014), pp.~1207--1216.

\bibitem{Dolgov2015}
{\sc S.~V. Dolgov and D.~V. Savostyanov}, {\em Corrected one-site density
  matrix renormalization group and alternating minimal energy algorithm}, in
  Numerical Mathematics and Advanced Applications - ENUMATH 2013, Springer,
  2015, pp.~335--343.

\bibitem{grasedyck2010hierarchical}
{\sc L.~Grasedyck}, {\em Hierarchical singular value decomposition of tensors},
  SIAM J. Matrix Anal. Appl., 31 (2010), pp.~2029--2054.

\bibitem{Grasedyck2013}
{\sc L.~Grasedyck, D.~Kressner, and C.~Tobler}, {\em A literature survey of
  low-rank tensor approximation techniques}, GAMM-Mitteilungen, 36 (2013),
  pp.~53--78.

\bibitem{hackbusch2012}
{\sc W.~Hackbusch, B.~N. Khoromskij, S.~Sauter, and E.~E. Tyrtyshnikov}, {\em
  Use of tensor formats in elliptic eigenvalue problems}, Numer. Linear Algebra
  Appl., 19 (2012), pp.~133--151.

\bibitem{Haegeman2016}
{\sc J.~Haegeman, C.~Lubich, I.~Oseledets, B.~Vandereycken, and F.~Verstraete},
  {\em Unifying time evolution and optimization with matrix product states},
  Phys. Rev. B, 94 (2016), p.~165116.

\bibitem{Hastings2007}
{\sc M.~B. Hastings}, {\em An area law for one-dimensional quantum systems}, J.
  Stat. Mech. Theory Exp., 2007 (2007), p.~P08024.

\bibitem{TT-mfld}
{\sc S.~Holtz, T.~Rohwedder, and R.~Schneider}, {\em On manifolds of tensors of
  fixed {TT}-rank}, Numer. Math., 120 (2012), pp.~701--731.

\bibitem{HuChan2015}
{\sc W.~Hu and G.~K.-L. Chan}, {\em Excited-state geometry optimization with
  the density matrix renormalization group, as applied to polyenes}, J. Chem.
  Theory Comput., 11 (2015), pp.~3000--3009.
\newblock PMID: 26575737.

\bibitem{MBLExcitedDMRG}
{\sc V.~Khemani, F.~Pollmann, and S.~L. Sondhi}, {\em Obtaining highly excited
  eigenstates of many-body localized hamiltonians by the density matrix
  renormalization group approach}, Phys. Rev. Lett., 116 (2016), p.~247204.

\bibitem{knyazev2012principal}
{\sc A.~V. Knyazev and P.~Zhu}, {\em Principal angles between subspaces and
  their tangents}, arXiv preprint arXiv:1209.0523,  (2012).

\bibitem{kolda2009tensor}
{\sc T.~G. Kolda and B.~W. Bader}, {\em Tensor decompositions and
  applications}, SIAM Rev., 51 (2009), pp.~455--500.

\bibitem{kressner2014}
{\sc D.~Kressner, M.~Steinlechner, and A.~Uschmajew}, {\em Low-rank tensor
  methods with subspace correction for symmetric eigenvalue problems}, SIAM J.
  Sci. Comput., 36 (2014), pp.~A2346--A2368.

\bibitem{kressner2010krylov}
{\sc D.~Kressner and C.~Tobler}, {\em Krylov subspace methods for linear
  systems with tensor product structure}, SIAM J. Matrix Anal. Appl., 31
  (2010), pp.~1688--1714.

\bibitem{kressner2011}
\leavevmode\vrule height 2pt depth -1.6pt width 23pt, {\em Preconditioned
  low-rank methods for high-dimensional elliptic pde eigenvalue problems},
  Comput. Methods Appl. Math., 11 (2011), pp.~363--381.

\bibitem{lu2020}
{\sc J.~Lu and Z.~Wang}, {\em The full configuration interaction quantum monte
  carlo method through the lens of inexact power iteration}, SIAM J. Sci.
  Comput., 42 (2020), pp.~B1--B29.

\bibitem{Lubich2015}
{\sc C.~Lubich, I.~V. Oseledets, and B.~Vandereycken}, {\em Time integration of
  tensor trains}, SIAM J. Numer. Anal., 53 (2015), pp.~917--941.

\bibitem{Orus2014}
{\sc R.~Or{\'u}s}, {\em A practical introduction to tensor networks: Matrix
  product states and projected entangled pair states}, Ann. Physics, 349
  (2014), pp.~117--158.

\bibitem{oseledets2011tensor}
{\sc I.~V. Oseledets}, {\em Tensor-train decomposition}, SIAM J. Sci. Comput.,
  33 (2011), pp.~2295--2317.

\bibitem{PAIGE1980235}
{\sc C.~Paige}, {\em Accuracy and effectiveness of the lanczos algorithm for
  the symmetric eigenproblem}, Linear Algebra Appl., 34 (1980), pp.~235--258.

\bibitem{Palitta2021}
{\sc D.~Palitta and P.~K{\"u}rschner}, {\em On the convergence of krylov
  methods with low-rank truncations}, Numer. Algorithms, 88 (2021),
  pp.~1383--1417.

\bibitem{parlett}
{\sc B.~N. Parlett}, {\em The Symmetric Eigenvalue Problem}, no.~20 in Classics
  in Applied Mathematics, SIAM, Philadelphia, 1998.

\bibitem{Saad2011}
{\sc Y.~Saad}, {\em Numerical methods for large eigenvalue problems: revised
  edition}, SIAM, 2011.

\bibitem{Saad_2016}
{\sc Y.~Saad}, {\em Analysis of subspace iteration for eigenvalue problems with
  evolving matrices}, SIAM J. Matrix Anal. Appl., 37 (2016), pp.~103--122.

\bibitem{Schollwock2005}
{\sc U.~Schollw{\"o}ck}, {\em The density-matrix renormalization group}, Rev.
  Modern Phys., 77 (2005), pp.~259--315.

\bibitem{Schollwock2011}
\leavevmode\vrule height 2pt depth -1.6pt width 23pt, {\em The density-matrix
  renormalization group in the age of matrix product states}, Ann. Physics, 326
  (2011), pp.~96--192.

\bibitem{Simon1984}
{\sc H.~D. Simon}, {\em Analysis of the symmetric lanczos algorithm with
  reorthogonalization methods}, Linear Algebra Appl., 61 (1984), pp.~101--131.

\bibitem{Steinlechner2016}
{\sc M.~Steinlechner}, {\em Riemannian optimization for solving
  high-dimensional problems with low-rank tensor structure}, tech. rep., EPFL,
  2016.

\bibitem{stewart1998perturbation}
{\sc G.~W. Stewart}, {\em Perturbation theory for the singular value
  decomposition}, Citeseer, 1998.

\bibitem{wang1992svBounds}
{\sc B.~Wang and F.~Zhang}, {\em Some inequalities for the eigenvalues of the
  product of positive semidefinite hermitian matrices}, Linear Algebra Appl.,
  160 (1992), pp.~113--118.

\bibitem{White1993}
{\sc S.~R. White}, {\em Density-matrix algorithms for quantum renormalization
  groups}, Phys. Rev. B, 48 (1993), p.~10345.

\bibitem{White2005}
\leavevmode\vrule height 2pt depth -1.6pt width 23pt, {\em Density matrix
  renormalization group algorithms with a single center site}, Phys. Rev. B, 72
  (2005), p.~180403.

\bibitem{yuan2013truncated}
{\sc X.-T. Yuan and T.~Zhang}, {\em Truncated power method for sparse
  eigenvalue problems.}, J. Mach. Learn. Res., 14 (2013).

\bibitem{zhai2023block2}
{\sc H.~Zhai, H.~R. Larsson, S.~Lee, Z.-H. Cui, T.~Zhu, C.~Sun, L.~Peng,
  R.~Peng, K.~Liao, J.~T{\"o}lle, et~al.}, {\em Block2: A comprehensive open
  source framework to develop and apply state-of-the-art dmrg algorithms in
  electronic structure and beyond}, The Journal of Chemical Physics, 159
  (2023).

\bibitem{spec_bound}
{\sc Y.~Zhou and R.-C. Li}, {\em Bounding the spectrum of large hermitian
  matrices}, Linear Algebra Appl., 435 (2011), pp.~480--493.
\newblock Special Issue: Dedication to Pete Stewart on the occasion of his 70th
  birthday.

\end{thebibliography}
